\documentclass[11pt]{article}
\usepackage{amsmath,amsthm,amsfonts,amssymb,amscd, amsxtra, mathrsfs,enumitem, mathtools}
\usepackage{url}

\usepackage[margin=2.3 cm,nohead]{geometry}
\usepackage{color}
\usepackage{pdfsync}
\usepackage{hyperref}
\usepackage{graphicx}
\usepackage{subcaption}
\usepackage{caption}
\usepackage{booktabs}    
\usepackage{float}  
\usepackage[
  backend=biber,
  style=ieee,
  citestyle=numeric-comp, 
  url=false,
  isbn=false,
  sortcites=true
]{biblatex}

\usepackage{multicol}
\usepackage{multirow}
\usepackage{siunitx}
\setlist[enumerate]{label=(\roman*), align=left}
\newcommand{\matM}{\mathbb M}
\newtheorem{theorem}{Theorem}
\newtheorem{lemma}[theorem]{Lemma}

\newtheorem{corollary}[theorem]{Corollary}
\newtheorem{proposition}[theorem]{Proposition}

\newtheorem{remark}{Remark}

\newtheorem{example}{Example}
\newtheorem{assumption}{Assumption}

\usepackage{graphicx,float}

\usepackage{algorithm}
\usepackage{algpseudocode}

\newcommand{\R}{\mathbb{R}}
\newcommand{\setS}{\mathcal{S}}
\newcommand{\inner}[2]{\langle #1, #2 \rangle}

\DeclareMathOperator{\diag}{diag}  
\DeclareMathOperator{\argmin}{argmin}  
\usepackage{amssymb}
\usepackage{relsize}

\begin{document}
\title{Convergence Rates for Variational Inequality Projection Neural Networks with a State-Dependent Metric}

\author{
Mohammed Alshahrani \thanks{Department of Mathematics, King Fahd University of Petroleum \& Minerals, Dhahran, 31261, Saudi Arabia\\
Interdisciplinary Research Center for Smart Mobility and Logistics, King Fahd University of Petroleum \& Minerals, Dhahran, 31261, Saudi Arabia, (e-mail:{\tt mshahrani@kfupm.edu.sa}).}
}

\maketitle

\begin{abstract}
We study continuous-time projection neural networks for variational inequalities
on closed convex sets. A positive definite matrix that depends on the state
preconditions the operator, and its inverse defines the projection
metric. Existing convergence analyses of this flow cover Hessian-generated
inverse metrics and state-dependent scalar metrics. In the first case, a Bregman
distance eliminates metric-derivative terms. We treat a matrix metric whose
inverse is of neither kind. We prove joint regularity of the projection in its
argument and metric. Under common spectral bounds, the Euclidean Lipschitz estimate
improves from the squared bound to the bound itself. For Lipschitz strongly
monotone operators and Lipschitz metrics, we prove local exponential
convergence with explicit rate and radius. On compact feasible sets, an explicit
metric-variation bound yields global exponential convergence. A twice
continuously differentiable, uniformly positive definite metric and a strongly
monotone linear operator produce an annulus of periodic orbits. The inverse
metric violates Hessian integrability throughout the annulus. Deterministic
computations confirm the analytic formulas and quantify slack in both
convergence certificates.
\end{abstract}

\noindent
{\bf Keywords:} projection neural networks; variational inequalities; state-dependent metric; variable-metric methods; exponential convergence; periodic orbits\\

\medskip

\noindent
{\bf AMS subject classification:}  37C27, 37N40, 47H05, 49J40, 90C33
%

\section{Introduction}

Variational inequality (VI) problems provide a unifying framework for a wide
range of equilibrium models arising in optimization, game theory, network
analysis, and constrained dynamical systems. Given a closed convex set
$\setS \subset \mathbb R^n$ and a mapping $F:\mathbb R^n \to \mathbb R^n$, the
classical variational inequality problem $\mathrm{VI}(F,\setS)$ consists of
finding a point $\bar{x} \in \setS$ such that
\[
\langle F(\bar{x}), x - \bar{x} \rangle \ge 0, \quad \forall x \in \setS.
\]
Projection-based methods and their continuous-time counterparts are a standard
route to such problems, particularly under monotonicity assumptions on $F$.

The continuous-time line is usually traced to the projected dynamical system
(PDS) of Dupuis and Nagurney \cite{Dupuis1993}. Write $P_{\setS}$ for Euclidean
projection onto $\setS$ and $\pi(x,v)$ for the directional derivative
$\lim_{\delta\downarrow0}\delta^{-1}(P_{\setS}(x+\delta v)-x)$,
which for closed convex $\setS$ is the Euclidean projection of $v$ onto the
tangent cone at $x$ \cite{Dupuis1993}. The PDS is
\[
\frac{dx}{dt} = \pi\bigl(x,-F(x)\bigr),
\]
and its stationary points are exactly the solutions of $\mathrm{VI}(F,\setS)$.
Its right-hand side is discontinuous, so classical ODE theory does not apply.
Existence and uniqueness were established through the Skorokhod problem
\cite{Dupuis1993,Nagurney1996}, and existence again later by an independent
argument that recasts the flow as a complementarity system \cite{Heemels2000}.
Either route needs an absolutely continuous solution satisfying the dynamics
almost everywhere, because the field is not continuous. The
same projected-flow idea
has since been carried to Hilbert spaces of arbitrary dimension and to
evolutionary variational inequalities \cite{Cojocaru2005}. Throughout
that work the metric is fixed in advance, even where generalized solutions or
nonsmooth analysis are needed.

Projection neural networks (PNNs) use a different field with the same equilibria
\cite{Hu2006,Xia2004},
\[
\frac{dx}{dt}= \lambda\bigl(P_{\setS}(x-\alpha F(x)) - x\bigr),
\]
where $\lambda,\alpha>0$ are fixed gains. Several formulations replace the scalar
$\lambda$ by a constant diagonal
matrix, or by a general square matrix, acting on the residual
\cite{Xia2004,Hu2007}. Such a matrix left-preconditions the residual and leaves
the projection Euclidean. The two flows
are not the same. At fixed $\alpha$ the residual $P_{\setS}(x-\alpha
F(x))-x$ is a chord into $\setS$, and the field is continuous, indeed locally
Lipschitz. In general the tangent-cone field is recovered as
$\alpha\downarrow0$ with $\lambda=1/\alpha$, which is how $\pi$ is defined in
\cite{Dupuis1993}. At a fixed $\alpha$ the two need not agree, though they do in
special geometries such as an affine $\setS$. Global convergence, Lyapunov
stability and exponential rates have been established for PNNs under several
monotonicity hypotheses, each strictly stronger than monotonicity itself
\cite{Xia2004,Hu2006,Hu2007}. Hu and
Wang exhibit a monotone affine operator on a box whose trajectories oscillate
periodically and never reach the solution \cite{Hu2006}, and
Section~\ref{sec:stabI} returns to that obstruction. The arguments that do
succeed rest on nonexpansiveness of the Euclidean projection and on energy
functions built in a fixed geometry.

The practical motivation for variable metrics comes from discrete algorithms.
When Euclidean geometry is poorly aligned with the operator or the feasible set,
scaled gradient projection uses positive step sizes $\alpha_k$ and
iteration-dependent positive definite matrices $M_k$ built from curvature
information or problem structure \cite{Bonettini2009}. Its
update has the form
\[
x_{k+1} = P_{\setS,M_k^{-1}}\bigl(x_k - \alpha_k M_k F(x_k)\bigr).
\]
That history motivates allowing the metric to vary with the state. The question
here is analytical rather than comparative. We ask which well-posedness and
convergence guarantees survive when the matched preconditioner and projection
metric in a continuous flow depend on the current point. We do not claim that this
state dependence accelerates the flow. Splitting and resolvent methods replace the
projection by the resolvent of a nonsmooth term, which changes the step rather
than the metric, and they do so at the level of the iteration \cite{Noor2001}.

Continuous-time models with non-Euclidean geometry are well established, and in
much of that work the geometry is generated by a convex potential. The
Hessian--Riemannian gradient flow $\dot x+\nabla^2\varphi(x)^{-1}\nabla f(x)=0$ takes
its metric from the Hessian of a Legendre function $\varphi$, whose gradient blows up at the
boundary of the feasible set and so keeps the trajectory inside
\cite{Alvarez2004,Attouch2004}. Mirror descent dynamics for monotone variational
inequalities integrate in the dual space and recover the state through the mirror
map of a distance-generating function \cite{Mertikopoulos2018}. Implicit dynamics
built from the resolvents of monotone operators form a large class, surveyed in
\cite{Csetnek2020}. Recent flows built for safety and invariance keep
the Euclidean inner product and modify instead the set that is projected onto,
replacing the tangent cone by a restricted tangent set that stays nonempty on a
neighbourhood of the feasible set \cite{Allibhoy2025}. A further variant uses a
state-dependent feasible set $K(x)$, giving
$\dot x=\lambda\bigl(P_{K(x)}(x-\alpha F(x))-x\bigr)$ for quasi-variational
inequalities \cite{Noor2003}.

Where a potential is present it is doing real work. It supplies the Lyapunov
functional, as a Bregman distance or a Fenchel coupling, and Alvarez, Bolte and
Brahic characterize when that is available. Their Theorem~3.1 says the
displacement fields $x\mapsto x-y$ are all gradients in a Riemannian metric
$H$ exactly when $H=\nabla^2\varphi$ for a strictly convex $\varphi$ of class $C^3$
\cite{Alvarez2004}. Off that class one may still have a Lyapunov functional, but
not one whose gradient is the displacement field, which is what makes the
metric-derivative term disappear.

The potential is not universal, however, and the work closest to ours does
without it. The closest is that of Hauswirth, Bolognani and
D\"orfler, who study projected dynamics under a pointwise Riemannian metric
\cite{Hauswirth2020}. Their metric is genuinely state dependent, and it makes the
projection directions oblique on irregular and possibly nonconvex domains. Their
projected field is the metric projection of the vector field onto the tangent
cone. On such domains that map can be multivalued or discontinuous, and their
general existence theory uses Krasovskii solutions. Under stronger regularity
they recover Carath\'eodory solutions and uniqueness. Optimization on Riemannian
manifolds also works outside the Euclidean metric, and there the iterates live on
the manifold itself rather than in a convex subset of $\R^n$. A metric is fixed
first, and the algorithmic work goes into realizing steps on the manifold through
retractions, vector transports and local models
\cite{AbsilMahonySepulchre2008,Ring2012}.

The map $P_{\setS,\matM(x)^{-1}}(x-\alpha\matM(x)F(x))$, with a
state-dependent positive definite metric matched to the preconditioner, has both
discrete- and continuous-time antecedents. In discrete time it appears in 1988 as
a variable-metric gradient projection process with a line search for a smooth
objective \cite{Gawande1988}. Antipin studies the continuous Euclidean
form $\dot x+x=P_{\setS}(x-\alpha\nabla f(x))$, deriving a descent estimate, a
step-size window determined by the gradient Lipschitz constant, and an explicit
exponential rate under strong convexity, while also treating infeasible initial
points \cite{Antipin1994}. He attributes the flow and its exponential stability to
earlier work of his own from 1989. It was later carried to Hilbert space
\cite{Bolte2003} and to monotone inclusions \cite{Abbas2015,Bot2018}.
Mijajlović and Jaćimović treat the more general quasi-variational setting in which
the feasible set also moves, and cite earlier continuous variable-metric methods
for minimization and variational inequalities \cite{Mijajlovic2018}.

The flow this paper analyses is therefore
\[
\frac{dx}{dt} = \lambda\bigl(P_{\setS,\matM(x)^{-1}}(x-\alpha \matM(x)F(x)) - x\bigr),
\]
written here as a state-dependent scaled projection neural network (SD-SPNN),
with $\setS$ a fixed closed convex subset of $\R^n$ and $\matM(\cdot)$ a
matrix-valued map taking positive definite values. The inner product
$\langle\cdot,\cdot\rangle_{\matM(x)}$ moves with the state while the feasible set
does not, so the geometry is state dependent without being intrinsic, and
different choices of $\matM(\cdot)$ give different flows on the same set.

What is open is the case of a metric that is neither a Hessian nor a multiple of
the identity. The convergence theorem of \cite{Mijajlovic2018} requires
$\matM(\cdot)^{-1}=\Phi''$ for a strongly convex $\Phi$, and its Lyapunov function
is the Bregman distance of that $\Phi$. The same hypothesis appears in
\cite{Amochkina1997}. A state-dependent scalar inverse metric
$\matM(x)^{-1}=\beta(x)I$ falls outside that hypothesis, because
$\nabla^2\Phi=\beta(x)I$ forces $\beta$ to be constant when $n\ge2$.
Mijajlović and Jaćimović obtain an exponential estimate for that scalar class
\cite{Mijajlovic2018}. It is not a formality. The Bregman functional is the object
that makes those proofs work, and Theorem~3.1 of \cite{Alvarez2004} identifies
Hessian metrics as exactly the class for which it is available with the
displacement field as its gradient. Although \cite{Mijajlovic2018} states
Theorem~4 as a convergence result, its proof yields
an exponential estimate. The distinction is the integrability hypothesis, not the
presence of an exponential estimate. Without a potential $\Phi$ satisfying
$\matM^{-1}=\Phi''$, the Bregman Lyapunov function used in the existing proofs is
unavailable. The weighted quadratic in the current metric then acquires a term in
$\dot{\matM}$ whose sign is not controlled by the standing assumptions.
Section~\ref{sec:stabIII} returns to this point, and
Section~\ref{sec:counterexample} constructs a $C^2$ metric for which the
metric-derivative term cancels the contraction on an annulus of periodic orbits.

Even when $\inner{F(x)}{x-\bar x}>0$, a positive definite matrix $B$ can satisfy
$\inner{BF(x)}{x-\bar x}=0$ at a prescribed point.

\begin{example}[Pointwise cancellation by a positive definite matrix]\label{ex:mechanism}
Let
\[
\setS=\{x\in\R^2:\|x\|\le1.1\},
\qquad
F(x)=(I+2J)x,
\qquad
J=\begin{pmatrix}0&-1\\1&0\end{pmatrix}.
\]
The matrix $J$ rotates vectors counterclockwise through $\pi/2$. The operator $F$
is $1$-strongly monotone in the sense of Assumption~\ref{ass:strong_monotone}, and
$\bar x=0$ is the unique solution. Write $e_1$ for
the first standard basis vector of $\R^2$. At $x=e_1$,
\[
F(e_1)=\begin{pmatrix}1\\2\end{pmatrix},
\qquad
B=\begin{pmatrix}1&-1/2\\-1/2&1/2\end{pmatrix}\succ0,
\qquad
BF(e_1)=\begin{pmatrix}0\\1/2\end{pmatrix}=\tfrac12Je_1 .
\]
\end{example}

Since $J^\top=-J$, one has $\inner{F(x)}{x}=\|x\|^2$ for every $x\ne0$. Thus the
Euclidean direction $-F(x)$ has a strictly negative radial component. At $x=e_1$,
however, $\inner{BF(e_1)}{e_1}=0$. For the choice $\alpha=0.05$ used in
Section~\ref{sec:counterexample}, the forward point has norm
$\sqrt{1+\alpha^2/4}<1.1$. The projection is inactive there, and
$\dot x=-(\lambda\alpha/2)Jx$ is tangent to the unit circle.

For a general $x=(x_1,x_2)$,
\[
\inner{BF(x)}{x}=x_2\Bigl(\tfrac32x_2-2x_1\Bigr).
\]
A fixed $B$ therefore cancels the radial component only on two lines. Moreover
$B(I+2J)$ has eigenvalues $3/4\pm i\sqrt{11}/4$, so the interior linear flow with
$B$ fixed is asymptotically stable. Section~\ref{sec:counterexample} instead sets
$\matM(x)=R(\theta)BR(\theta)^\top$, where $\theta$ is the polar angle of $x$ and
$R(\theta)$ is the rotation carrying $e_1$ to $x/\|x\|$. On the resulting annulus,
$\matM(x)F(x)=\tfrac12Jx$ at every point. Every trajectory starting on one of
those circles is periodic and does not converge to $\bar x$.

Discrete variable-metric methods allow free metrics, but their hypotheses do not
resolve the state-dependent continuous-time problem. Variable-metric
forward--backward splitting runs on
an arbitrary sequence of positive definite operators, with no potential anywhere,
under a summable bound on how far the metric may move between steps
\cite{Combettes2014}. A second route drops that bound and imposes a majorization
condition at the current iterate instead \cite{Chouzenoux2014}. For monotone
inclusions, a two-sided bound of the same summable kind, together with a
regularity condition at the solution, already yields a linear rate
\cite[Theorem~2.2]{Lotito2009}. These results provide metric freedom, variation
budgets and rates under such budgets in discrete time.

What that literature does not do is differentiate the metric with respect to the
state. In \cite{Combettes2014} the metric is an exogenous sequence, and a frozen
metric is the Hessian of a quadratic \cite[Remark~3.8]{Combettes2014}. In
\cite{Chouzenoux2014} it is tied to the iterate, but only through a pointwise
inequality that a constant multiple of the identity already satisfies
\cite[Lemma~3.1]{Chouzenoux2014}. In \cite{Lotito2009} the matrices are chosen
freely at each step, and the budget constrains consecutive terms of a sequence
rather than a derivative. A flow has no such option. The derivative of a weighted
energy along a trajectory
carries $\dot{\matM}$ whether the metric is integrable or not, and that term is
what the rest of this paper must control.

Making the metric depend on the state changes the structure of the problem, not
only its constants. The projection operator itself becomes state dependent, and
any Lyapunov argument must carry the time variation of the metric along the
trajectory. The analysis below separates what survives that change from what
does not. Well-posedness survives. Local Lipschitz continuity of $F$, closed
convexity of $\setS$, uniform bounds on the metric and Lipschitz regularity of
$\matM(\cdot)$ give a locally Lipschitz residual field and a unique viable local
solution, and forward completeness then needs only a growth or continuation
argument. Stability does not survive unchanged. Metric variation contributes a
term of indefinite sign to the derivative of a moving-metric energy. The standing
assumptions alone do not control that term and therefore do not imply global
convergence. A quantitative variation bound is one sufficient way to recover
global decay.

Within this model, Euclidean projection neural networks are the case
$\matM(x)\equiv I$. The classical projected dynamical system is not: it is the
tangent-cone flow, recovered only in the limit described above. Constant-metric
instances $\matM(x)\equiv M\succ0$ are continuous-time analogues of scaled
projection methods. Relative to the tangent-cone dynamics of
\cite{Hauswirth2020}, the metric here sits inside a finite forward step rather
than in a projection onto a cone. That form keeps an exact equivalence between
equilibria and solutions of $\mathrm{VI}(F,\setS)$ at every $\alpha>0$.

A word on the name. The SD-SPNN inherits the architecture of classical
projection neural networks, in which the dynamics are realized by units
implementing linear combinations, activations and projection blocks. In that
tradition ``neural network'' means an analog or neuromorphic realization of a
flow that solves an optimization or equilibrium problem in continuous time,
rather than a model trained on data. The state-dependent metric is then a
continuously updated gain in the circuit, changing how the residual
$x-P_{\setS,\matM(x)^{-1}}(x-\alpha \matM(x)F(x))$ is formed and fed back. All
results below are stated and proved at the level of dynamical systems, with no
learning and no data. Adjacent developments include further projection neural
networks for variational inequalities \cite{Jiang2012,Xu2021}, fixed-time and
predefined-time neurodynamic solvers
\cite{Yang2025,Zheng2024,Tran2025}, systems for generalized and inverse
variational inequalities \cite{Anh2025}, and networks analysed through the
variational inequality their asymptotic output solves
\cite{Combettes2020,Combettes2022,Zhao2025a}.

\paragraph{Contributions.}
\begin{itemize}
\item Joint regularity of the metric projection in its argument and in its
metric, with explicit constants, including a Euclidean Lipschitz constant $L$
rather than the $L^2$ available in the literature for the same class of matrices.
\item Local exponential convergence for a state-dependent metric that is not
assumed to be a Hessian or a multiple of the identity. We give an explicit rate
and certified radius, and obtain a global rate under a bound on metric variation.
\item A counterexample delimiting all of this. On a compact feasible set, with a
strongly monotone operator, a $C^2$ uniformly positive definite metric whose
eigenframe rotates, and $D\matM(\bar x)=0$, the flow has an annulus of periodic
orbits. The construction shows that the standing assumptions alone do not imply
global convergence. A metric-variation budget is one explicit sufficient condition
that excludes this behaviour.
\item A numerical study of the counterexample and of the two certificates.
Computed trajectories on the periodic annulus match the analytic orbits to
$4.3\times10^{-13}$ in radius and $3.5\times10^{-13}$ in phase. The integrability
residual is validated against metrics whose residual is known in closed form, and
the projection estimates of Section~\ref{sec:projreg} are evaluated at an active
constraint. The half-rate radius is $4.4661$ decades below the analytic boundary
between initial states that converge to $\bar x$ and periodic initial states,
measured in the weighted radius of that theorem. The explicit variation
certificate covers a band $7.1763$ times shorter than the true one. The
computations are deterministic.
\end{itemize}

The remainder of the paper is organized as follows.
Section~\ref{sec:notation} introduces notation, definitions and the standing
assumptions. Section~\ref{sec:projreg} establishes the regularity of the metric
projection in its argument and in its metric, which is what the later arguments
rest on. Section~\ref{sec:well-posed} establishes well-posedness, invariance of the
feasible set and the equilibrium characterization.
Section~\ref{sec:stability} develops the stability analysis in three stages,
first conditionally, then for a fixed metric, then for a genuinely
state-dependent one. Section~\ref{sec:counterexample} gives the counterexample that
delimits all of it, together with an integrable metric for contrast.
Section~\ref{sec:numerics} reports a numerical study of the counterexample, the
two explicit certificates and the projection estimates at an active constraint.
Section~\ref{sec:futurework} closes with limitations and directions for further
work.

\section{Formulation and Notation}\label{sec:notation}

Let $\setS \subset \R^n$ be a nonempty closed convex set and let
$F : \R^n \to \R^n$ be a continuous mapping. The variational inequality
problem associated with $(F,\setS)$, denoted by $\mathrm{VI}(F,\setS)$,
consists of finding a point $\bar x \in \setS$ such that
\begin{equation}\label{eq:VI}
    \inner{F(\bar x)}{x - \bar x} \ge 0,
    \quad \forall x \in \setS.
\end{equation}
We write $\operatorname{SOL}(F,\setS)$ for the set of such $\bar x$.

Throughout the paper, $\inner{\cdot}{\cdot}$ denotes the standard
Euclidean inner product on $\R^n$, with associated norm
$\|x\| = \sqrt{\inner{x}{x}}$.
Given a symmetric positive definite matrix $M \in \R^{n \times n}$,
we define the $M$--induced inner product and norm by
\[
\inner{x}{y}_M := x^T M y,
\qquad
\|x\|_M := \sqrt{\inner{x}{x}_M}.
\]

For matrices, $\|\cdot\|$ denotes the Euclidean operator norm, $I$ the identity
matrix, and $\lambda_{\min}(A)$ and $\lambda_{\max}(A)$ the extreme eigenvalues
of a symmetric matrix $A$. For symmetric matrices $A$ and $B$, the relation
$A\preceq B$ means that $B-A$ is positive semidefinite, while $A\succ0$ means
that $A$ is positive definite. We write $\diag(a_1,\dots,a_n)$ for the diagonal
matrix with those entries. Finally, $\overline B(x,r)$ denotes the closed
Euclidean ball, and
$\operatorname{dist}(x,\mathcal A):=\inf_{z\in\mathcal A}\|x-z\|$.

We consider a state-dependent metric generated by a mapping
$\matM : \R^n \to \R^{n \times n}$ satisfying
\[
\left[\matM(x)\right]^T = \matM(x),
\qquad
y^T \matM(x) y > 0 \quad \forall y \neq 0,
\]
for all $x \in \R^n$.
The corresponding inner product $\inner{\cdot}{\cdot}_{\matM(x)}$
and norm $\|\cdot\|_{\matM(x)}$ are referred to as
\emph{state-dependent}.

\subsection*{Metric projection}

For a fixed $x \in \R^n$, we define the projection onto $\setS$
with respect to the metric induced by $\matM(x)$ as
\begin{equation}\label{eq:metric_projection}
    P_{\setS,\matM(x)}(y)
    :=
    \argmin_{z \in \setS}
    \frac{1}{2}\|z - y\|^2_{\matM(x)}.
\end{equation}
Since $\matM(x)$ is symmetric positive definite and $\setS$ is closed
and convex, the minimizer in \eqref{eq:metric_projection} exists and
is unique. If $\matM(x)\equiv I$, then \eqref{eq:metric_projection} is the
Euclidean projection; if $\matM(x)\equiv M$ is constant, it is commonly called a
scaled projection \cite{Bonettini2009}.

The metric projection has the following variational characterization.
\begin{lemma}[Variational characterization of the metric projection]\label{lem:proj_char}
For given $x,y \in \R^n$,  let $z \in \setS$. Then
\begin{equation}\label{eq:proj_char}
    z = P_{\setS,\matM(x)}(y)
    \quad \Longleftrightarrow \quad
    (z - y)^T \matM(x)(w - z) \ge 0,
    \quad \forall w \in \setS.
\end{equation}
\end{lemma}
\begin{proof}
Let $f$ be the objective minimized in \eqref{eq:metric_projection},
\[
f(u)=\tfrac12 (u-y)^\top \matM(x)(u-y),
\]
so that $P_{\setS,\matM(x)}(y)=\argmin_{u\in\setS}f(u)$.
Then $f$ is convex and $\nabla f(u)=\matM(x)(u-y)$. If $z=P_{\setS,\matM(x)}(y)$, then for any $w\in\setS$ and $t\in[0,1]$,
$z+t(w-z)\in\setS$. Define $g(t)=f(z+t(w-z))$. Since $t=0$ is a
minimizer of $g$,
\[
0\le g'(0^+)=\nabla f(z)^\top (w-z)=(z-y)^\top \matM(x)(w-z).
\]

Conversely, suppose $(z-y)^\top \matM(x)(w-z)\ge 0$ for all $w\in\setS$. Then for any $w\in\setS$,
\[
\begin{aligned}
f(w)-f(z)
&=\tfrac12 (w-y)^\top \matM(x)(w-y)-\tfrac12 (z-y)^\top \matM(x)(z-y)\\
&=\tfrac12 (w-z)^\top \matM(x)(w-z)+(z-y)^\top \matM(x)(w-z)\ge 0.
\end{aligned}
\]
Hence $z$ minimizes $f$ over $\setS$, i.e., $z=P_{\setS,\matM(x)}(y)$.
\end{proof}

\subsection*{Standing assumptions}

The assumptions used below are collected here. Each result states those it
requires.

\begin{assumption}[Uniform metric boundedness]\label{ass:metric_bounds}
There exists a constant $L \ge 1$ such that
\[
\|\matM(x)\| \le L
\quad \text{and} \quad
\|\matM(x)^{-1}\| \le L,
\qquad \forall x \in \setS.
\]
\end{assumption}

Assumption~\ref{ass:metric_bounds} implies that the eigenvalues of
$\matM(x)$ lie in $[L^{-1},L]$ uniformly on $\setS$, and hence all norms
$\|\cdot\|_{\matM(x)}$ are uniformly equivalent to the Euclidean norm. In fact, we have 
\[
\frac{1}{\sqrt{L}}\|z\|\le \|z\|_{\matM(x)}\le \sqrt{L} \|z\|, \quad \forall~z\in \R^n.
\]

From Section~\ref{sec:well-posed} onwards the metric must be regular and not
merely bounded.

\begin{assumption}[Metric regularity]\label{ass:metric_regular}
The mapping $\matM(\cdot)$ is locally Lipschitz continuous on $\setS$.
\end{assumption}

The remaining assumptions quantify the regularity of $\matM$ and restrict $F$.
Assumption~\ref{ass:M_Lipschitz} is used from Section~\ref{sec:projreg} onwards
and Assumption~\ref{ass:F_Lipschitz} from Section~\ref{sec:well-posed} onwards.
The two monotonicity assumptions are used only in
Section~\ref{sec:stability}.

\begin{assumption}[Regularity of $F$]\label{ass:F_Lipschitz}
The mapping $F:\R^n\to\R^n$ is Lipschitz continuous on $\setS$
with constant $K>0$, i.e.
\[
\|F(x)-F(y)\| \le K\|x-y\|,
\qquad \forall x,y\in\setS.
\]
\end{assumption}

\begin{assumption}[Metric Lipschitz continuity]\label{ass:M_Lipschitz}
The metric mapping $\matM:\R^n\to\R^{n\times n}$ is Lipschitz
continuous on $\setS$ with constant $K_M\ge 0$:
\[
\|\matM(x)-\matM(y)\| \le K_M\|x-y\|,
\qquad \forall x,y\in\setS.
\]
\end{assumption}

\begin{assumption}[Monotonicity of $F$]\label{ass:monotone}
The mapping $F:\R^n\to\R^n$ is monotone on $\setS$, i.e.
\[
\inner{F(x)-F(y)}{x-y} \ge 0,
\qquad \forall x,y \in \setS.
\]
\end{assumption}

\begin{assumption}[Strong monotonicity]\label{ass:strong_monotone}
The mapping $F:\R^n\to\R^n$ is strongly monotone on $\setS$ with modulus
$\mu>0$, i.e.
\[
\inner{F(x)-F(y)}{x-y} \ge \mu\|x-y\|^2,
\qquad \forall x,y\in\setS.
\]
\end{assumption}

If $\setS$ contains two distinct points, Assumptions~\ref{ass:F_Lipschitz}
and~\ref{ass:strong_monotone} force $\mu\le K$. Together with $L\ge1$ this gives
$\mu\le LK$, which is used in Section~\ref{sec:stabIII} to ensure that a
contraction factor defined there is real.

\subsection*{State-dependent projection neural network}

For a given $\alpha>0$, define the state-dependent projected map
\begin{equation}\label{eq:Tmap}
T(x)
:=
P_{\setS,\matM(x)^{-1}}\bigl(x-\alpha \matM(x)F(x)\bigr),
\qquad x \in \setS.
\end{equation}
The map $T$ is one forward step along $-\matM(x)F(x)$, projected back onto $\setS$
in the metric $\matM(x)^{-1}$. Motivated by the fixed-metric projection neural
networks commonly used for solving variational inequalities, we study the
state-dependent projection neural network
\begin{equation}\label{eq:SD_dynamics}
    \frac{dx}{dt} = \lambda\bigl(T(x)-x\bigr),
\end{equation}
where $\lambda>0$ is a given parameter. When $\matM(x)=I$, system
\eqref{eq:SD_dynamics} is precisely the classical projected neural network. Its
equilibria are the fixed points of $T$, which
Proposition~\ref{prop:equilibrium_VI} identifies with the solutions of
$\mathrm{VI}(F,\setS)$.

\section{Regularity of the metric projection}\label{sec:projreg}

The map $T$ depends on the state through both the point being projected and the
metric in which the projection is taken. The projection estimates below use only
convexity and the spectral bounds of Assumption~\ref{ass:metric_bounds}. They
involve neither $F$ nor the differential equation, so
Section~\ref{sec:well-posed} may invoke them without circularity.
Corollary~\ref{cor:proj_constants} adds Assumption~\ref{ass:M_Lipschitz}, to bound
how fast the inverse metric moves.

Throughout, $Q,Q_1,Q_2$ denote symmetric positive definite matrices.

\begin{lemma}[Variation of the metric at a fixed point]
\label{lem:proj_metric_var}
Let $\setS$ be nonempty, closed and convex, and let $Q_1,Q_2\succ0$. Then for
every $z\in\R^n$,
\begin{equation}\label{eq:proj_metric_var}
\bigl\|P_{\setS,Q_1}(z)-P_{\setS,Q_2}(z)\bigr\|
\;\le\;
\frac{\|Q_1-Q_2\|}{\lambda_{\min}(Q_1)}\,
\bigl\|P_{\setS,Q_2}(z)-z\bigr\|.
\end{equation}
\end{lemma}

\begin{proof}
Write $y_i:=P_{\setS,Q_i}(z)$ and $d:=y_1-y_2$. If $d=0$ the right-hand side of
\eqref{eq:proj_metric_var} is nonnegative and there is nothing to prove, so assume
$d\ne0$. By \eqref{eq:proj_char},
\[
\inner{Q_1(y_1-z)}{y_2-y_1}\ge0,
\qquad
\inner{Q_2(y_2-z)}{y_1-y_2}\ge0 .
\]
Adding the two inequalities gives
$\inner{Q_1(y_1-z)-Q_2(y_2-z)}{d}\le0$. Substituting
$Q_1(y_1-z)-Q_2(y_2-z)=Q_1d+(Q_1-Q_2)(y_2-z)$ yields
\[
\inner{Q_1 d}{d}
\;\le\;
-\inner{(Q_1-Q_2)(y_2-z)}{d}
\;\le\;
\|Q_1-Q_2\|\,\|y_2-z\|\,\|d\| .
\]
Since $\inner{Q_1d}{d}\ge\lambda_{\min}(Q_1)\|d\|^2$, dividing by
$\lambda_{\min}(Q_1)\|d\|>0$ gives \eqref{eq:proj_metric_var}.
\end{proof}

The estimate is anchored on the displacement $\|P_{\setS,Q_2}(z)-z\|$ rather than
on a diameter. This matters later, because that displacement is exactly what the
projection bound of Lemma~\ref{lem:displacement} controls. Note also that the
divisor $\lambda_{\min}(Q_1)$ and the anchor $P_{\setS,Q_2}$ carry opposite
indices, and that they may be exchanged only together.

\begin{lemma}[Variation of the point at a fixed metric]
\label{lem:proj_arg_var}
Let $\setS$ be nonempty, closed and convex and let $Q\succ0$. Then
\begin{equation}\label{eq:proj_arg_var}
\bigl\|P_{\setS,Q}(z_1)-P_{\setS,Q}(z_2)\bigr\|_Q\le\|z_1-z_2\|_Q,
\qquad
\bigl\|P_{\setS,Q}(z_1)-P_{\setS,Q}(z_2)\bigr\|
\le
\sqrt{\frac{\lambda_{\max}(Q)}{\lambda_{\min}(Q)}}\;\|z_1-z_2\| .
\end{equation}
\end{lemma}

\begin{proof}
Put $y_i:=P_{\setS,Q}(z_i)$, $d:=y_1-y_2$ and $\delta:=z_1-z_2$. Using
\eqref{eq:proj_char} twice and adding,
\[
\inner{Q(y_1-z_1)-Q(y_2-z_2)}{y_2-y_1}\ge0,
\]
which rearranges to $\|d\|_Q^2\le\inner{\delta}{d}_Q\le\|\delta\|_Q\|d\|_Q$. This
gives the first inequality. The second follows from
$\sqrt{\lambda_{\min}(Q)}\,\|d\|\le\|d\|_Q$ and
$\|\delta\|_Q\le\sqrt{\lambda_{\max}(Q)}\,\|\delta\|$.
\end{proof}

\begin{proposition}[Joint regularity of the metric projection]
\label{prop:proj_joint}
Let $\setS$ be nonempty, closed and convex, and let $Q_1,Q_2\succ0$. Then for all
$z_1,z_2\in\R^n$,
\begin{equation}\label{eq:proj_joint}
\bigl\|P_{\setS,Q_1}(z_1)-P_{\setS,Q_2}(z_2)\bigr\|
\;\le\;
\sqrt{\frac{\lambda_{\max}(Q_1)}{\lambda_{\min}(Q_1)}}\;\|z_1-z_2\|
\;+\;
\frac{\|Q_1-Q_2\|}{\lambda_{\min}(Q_1)}\,
\bigl\|P_{\setS,Q_2}(z_2)-z_2\bigr\| .
\end{equation}
\end{proposition}

\begin{proof}
Insert the intermediate point $P_{\setS,Q_1}(z_2)$, which shares its metric with
the first term and its argument with the second, and apply
Lemma~\ref{lem:proj_arg_var} to
$\|P_{\setS,Q_1}(z_1)-P_{\setS,Q_1}(z_2)\|$ and
Lemma~\ref{lem:proj_metric_var} at $z=z_2$ to
$\|P_{\setS,Q_1}(z_2)-P_{\setS,Q_2}(z_2)\|$.
\end{proof}

\begin{remark}[The sharp constant]\label{rem:proj_constants}
The second constant in \eqref{eq:proj_arg_var} is not attained. Abbreviate
$\lambda_{\min}:=\lambda_{\min}(Q)$, $\lambda_{\max}:=\lambda_{\max}(Q)$ and
$\kappa:=\kappa_2(Q)=\lambda_{\max}/\lambda_{\min}$, the spectral condition number
of $Q$. The exact Euclidean modulus of
$P_{\setS,Q}$, taken over all nonempty closed convex $\setS$, is the Kantorovich
constant
\[
\gamma(Q)
:=\frac{\lambda_{\max}+\lambda_{\min}}{2\sqrt{\lambda_{\max}\lambda_{\min}}}
=\tfrac12\bigl(\sqrt{\kappa}+\kappa^{-1/2}\bigr).
\]
For the bound, keep $d$ and $\delta$ from the proof of
Lemma~\ref{lem:proj_arg_var} and start from $\|d\|_Q^2\le\inner{\delta}{d}_Q$
established there. Cauchy--Schwarz in the Euclidean norm gives
$\inner{\delta}{d}_Q=\inner{\delta}{Qd}\le\|\delta\|\,\|Qd\|$, so
\[
\frac{\|d\|}{\|\delta\|}\;\le\;\frac{\|d\|\,\|Qd\|}{\inner{Qd}{d}} .
\]
Every eigenvalue of $Q$ satisfies
$(\lambda-\lambda_{\min})(\lambda_{\max}-\lambda)\ge0$, so
$Q^2\preceq(\lambda_{\max}+\lambda_{\min})Q-\lambda_{\max}\lambda_{\min}I$.
Writing $s:=\inner{Qd}{d}/\|d\|^2\in[\lambda_{\min},\lambda_{\max}]$, this bounds
the right-hand side by
$\sqrt{(\lambda_{\max}+\lambda_{\min})s-\lambda_{\max}\lambda_{\min}}\,/\,s$,
which is largest at the harmonic mean
$s=2\lambda_{\max}\lambda_{\min}/(\lambda_{\max}+\lambda_{\min})$, where it equals
$\gamma(Q)$. The value is attained. Take $\setS=\R v$ a line, so that
$P_{\setS,Q}z=v\,\inner{Qz}{v}/\inner{Qv}{v}$ has Euclidean operator norm
$\|v\|\,\|Qv\|/\inner{Qv}{v}$, and choose the direction $v$ for which
$\inner{Qv}{v}/\|v\|^2$ is that harmonic mean.

The improvement is a factor $2\kappa/(\kappa+1)$. It is below two for every
$\kappa$ and approaches two as $\kappa\to\infty$. We keep the square-root form
because it is what the subsequent estimates combine with, and because the sharper
constant changes no conclusion below.
\end{remark}

Specializing to the metric of the dynamics gives the two constants used
throughout the paper.

\begin{corollary}[Constants under Assumption~\ref{ass:metric_bounds}]
\label{cor:proj_constants}
Let Assumption~\ref{ass:metric_bounds} hold and put $Q(x):=\matM(x)^{-1}$. Then
$\lambda_{\min}(Q(x))\ge L^{-1}$ and $\lambda_{\max}(Q(x))\le L$, so
\begin{equation}\label{eq:proj_L}
\bigl\|P_{\setS,Q(x)}(z_1)-P_{\setS,Q(x)}(z_2)\bigr\|\le L\,\|z_1-z_2\| .
\end{equation}
If in addition Assumption~\ref{ass:M_Lipschitz} holds, then
\begin{equation}\label{eq:inv_metric_var}
\bigl\|\matM(x)^{-1}-\matM(y)^{-1}\bigr\|\le L^2K_M\|x-y\|,
\qquad \forall x,y\in\setS .
\end{equation}
\end{corollary}

\begin{proof}
The spectral bounds are immediate from Assumption~\ref{ass:metric_bounds}, and
\eqref{eq:proj_L} follows from \eqref{eq:proj_arg_var} since
$\lambda_{\max}(Q)/\lambda_{\min}(Q)\le L^{2}$. For \eqref{eq:inv_metric_var} use
\[
\matM(x)^{-1}-\matM(y)^{-1}
=\matM(x)^{-1}\bigl(\matM(y)-\matM(x)\bigr)\matM(y)^{-1}
\]
together with Assumptions~\ref{ass:metric_bounds} and \ref{ass:M_Lipschitz}.
\end{proof}

\begin{remark}\label{rem:proj_constant_comparison}
The constant in \eqref{eq:proj_L} is $L$ and not $L^2$. Bonettini, Zanella and
Zanni prove the same Lipschitz bound with constant $L^2$, for the same class of
matrices, namely those with $\|D\|\le L$ and $\|D^{-1}\|\le L$, which is
Assumption~\ref{ass:metric_bounds} \cite[Lemma 2.1]{Bonettini2009}. The saving
comes from where the norms are converted. Bounding in the Euclidean norm
throughout costs $\|D\|$ at one end and $\|D^{-1}\|$ at the other, whereas
\eqref{eq:proj_arg_var} is an exact nonexpansiveness statement in $\|\cdot\|_Q$,
so only the single conversion $\sqrt{\lambda_{\max}(Q)/\lambda_{\min}(Q)}$ is
paid. Neither constant is sharp. By Remark~\ref{rem:proj_constants} the exact
modulus is at most $\tfrac12(L+L^{-1})$ under Assumption~\ref{ass:metric_bounds},
and that value is attained on a line when the spectral bounds are tight. We keep
$L$ because it is what the later estimates combine with.
The constant in \eqref{eq:inv_metric_var} is $L^2K_M$, and its exponent cannot be
lowered. In one dimension take $L>1$, $m_1=L^{-1}$ and $m_2=L^{-1}+\varepsilon$
with $0<\varepsilon<L-L^{-1}$, so that both lie in $[L^{-1},L]$. Then
\[
\frac{|m_1^{-1}-m_2^{-1}|}{|m_1-m_2|}=\frac{1}{m_1m_2}
=\frac{L^2}{1+\varepsilon L}\longrightarrow L^2
\qquad\text{as }\varepsilon\downarrow0 ,
\]
so both metric values approach the lower spectral bound and both inverse factors
approach $L$.
\end{remark}

\section{Well-posedness and Equilibria}\label{sec:well-posed}

In this section, we establish basic properties of the
state-dependent projection dynamical system
\eqref{eq:SD_dynamics}, including existence and uniqueness
of trajectories and a precise characterization of its equilibria. The right-hand
side of \eqref{eq:SD_dynamics} is $\lambda(T(x)-x)$ with $T$ as in
\eqref{eq:Tmap}, so every statement below is a statement about $T$.

\begin{lemma}[Continuity of the right-hand side]\label{lem:T_cont}
Let Assumptions~\ref{ass:metric_bounds} and~\ref{ass:metric_regular} hold, and let
$F$ be continuous on $\setS$. Then $T:\setS\to\setS$ is continuous.
\end{lemma}

\begin{proof}
By Assumption~\ref{ass:metric_regular} the mapping $x\mapsto\matM(x)$ is locally
Lipschitz on $\setS$, hence continuous, and so is $x\mapsto\matM(x)^{-1}$ under
Assumption~\ref{ass:metric_bounds}. Since $F$ is continuous, the forward point
$y(x):=x-\alpha\matM(x)F(x)$ depends continuously on $x$.

Fix $x_2\in\setS$ and apply Proposition~\ref{prop:proj_joint} with
$Q_i=\matM(x_i)^{-1}$ and $z_i=y(x_i)$. Its first term is at most
$L\|y(x_1)-y(x_2)\|$ by \eqref{eq:proj_L}, and its second is at most
$L\,\|\matM(x_1)^{-1}-\matM(x_2)^{-1}\|$ times the fixed finite number
$\|T(x_2)-y(x_2)\|$, using $\lambda_{\min}(\matM(x_1)^{-1})\ge L^{-1}$. Both
vanish as $x_1\to x_2$, so $T$ is continuous at $x_2$. Only the anchor at $x_2$
enters, so no bound on the displacement is needed.
\end{proof}

Two elementary bounds are recorded next. Both follow from the single observation
that $x$ itself is an admissible comparator in the projection defining $T(x)$.

\begin{lemma}[Projection displacement and descent]\label{lem:displacement}
Let Assumption~\ref{ass:metric_bounds} hold. Then for every $x\in\setS$,
\begin{equation}\label{eq:displacement}
\|T(x)-x\|_{\matM(x)^{-1}}\le\alpha\|F(x)\|_{\matM(x)},
\qquad
\|T(x)-x\|\le\alpha L\|F(x)\| ,
\end{equation}
and
\begin{equation}\label{eq:descent}
\inner{F(x)}{T(x)-x}
\le-\frac{1}{\alpha}\|T(x)-x\|^2_{\matM(x)^{-1}}
\le-\frac{1}{\alpha L}\|T(x)-x\|^2 .
\end{equation}
The second bound in \eqref{eq:displacement} is attained, for instance when
$\setS=\R^n$ and $\matM\equiv LI$.
\end{lemma}

\begin{proof}
Write $Q:=\matM(x)^{-1}$ and $y:=x-\alpha\matM(x)F(x)$, so that $T(x)=P_{\setS,Q}(y)$.
Since $x\in\setS$ we have $P_{\setS,Q}(x)=x$, and
Lemma~\ref{lem:proj_arg_var} gives
$\|T(x)-x\|_Q\le\|y-x\|_Q=\alpha\|\matM(x)F(x)\|_Q=\alpha\|F(x)\|_{\matM(x)}$,
using $\matM Q\matM=\matM$. Two applications of
Assumption~\ref{ass:metric_bounds} convert this to the Euclidean bound.

For \eqref{eq:descent}, apply \eqref{eq:proj_char} with $w=x$ and use
$Q\matM(x)=I$ to obtain
$\inner{Q(T(x)-x)+\alpha F(x)}{x-T(x)}\ge0$, which rearranges to
$\alpha\inner{F(x)}{T(x)-x}\le-\|T(x)-x\|_Q^2$.
\end{proof}

\begin{remark}\label{rem:descent_known}
At $\matM\equiv I$, $\lambda=1$ and $F=\nabla\varphi$ the first inequality in
\eqref{eq:descent} is the descent estimate
$\alpha\frac{d}{dt}\varphi(x(t))+\|\dot x(t)\|^2\le0$ of \cite{Bolte2003}, and
the derivation is the same one, namely taking $x$ itself as the comparator in the
projection inequality.
\end{remark}

\begin{lemma}[Local existence and viability]\label{lem:existence}
Under the assumptions of Lemma~\ref{lem:T_cont}, for every $x_0\in\setS$ there
exist $\tau>0$ and an absolutely continuous $x:[0,\tau]\to\setS$ with $x(0)=x_0$
solving \eqref{eq:SD_dynamics}.
\end{lemma}

\begin{proof}
The classical Peano theorem is stated for a vector field continuous on an open
subset of $\R^n$, whereas the right-hand side of \eqref{eq:SD_dynamics} is defined
only on $\setS$. At a boundary point of $\setS$ no ambient neighbourhood is
available, and when $\operatorname{int}\setS=\varnothing$ none is available at any
point. We argue instead through the variation-of-constants form of
\eqref{eq:SD_dynamics}. A locally absolutely
continuous $x$ with $x(0)=x_0$ solves \eqref{eq:SD_dynamics} if and only if
\begin{equation}\label{eq:varconst}
x(t)=e^{-\lambda t}x_0+\int_0^t\lambda e^{-\lambda(t-s)}T(x(s))\,ds ,
\end{equation}
as one checks by multiplying \eqref{eq:SD_dynamics} by $e^{\lambda t}$ and
integrating.

Fix $r>0$ and let $B_r:=\setS\cap\overline B(x_0,r)$, which is compact. Put
$N_r:=\max_{x\in B_r}\|T(x)-x_0\|$, finite by continuity of $T$ on $B_r$, and let
$\mathcal C$ be the set of continuous $u:[0,\tau]\to\setS$ with $u(0)=x_0$ and
$\|u-x_0\|_\infty\le r$. The set $\mathcal C$ is nonempty, closed, bounded and
convex. Let $A$ denote the right-hand side of \eqref{eq:varconst} with $u$ in
place of $x$. Since
\[
e^{-\lambda t}+\int_0^t\lambda e^{-\lambda(t-s)}\,ds=1
\]
with both weights nonnegative, $(Au)(t)$ is a convex combination of $x_0$ and of
points $T(u(s))\in\setS$, so $(Au)(t)\in\setS$ by convexity and closedness of
$\setS$. Moreover $\|(Au)(t)-x_0\|\le(1-e^{-\lambda\tau})N_r$, so $A$ maps
$\mathcal C$ into itself once $(1-e^{-\lambda\tau})N_r\le r$. If $T$ is Lipschitz
on $B_r$ with constant $\ell_r$, then
$\|Au-Av\|_\infty\le(1-e^{-\lambda\tau})\ell_r\|u-v\|_\infty$, so $A$ is a
contraction for $\tau$ small and Banach's theorem applies. If $F$ is merely
continuous, $A(\mathcal C)$ is equi-Lipschitz by \eqref{eq:varconst} and
Arzel\`a--Ascoli together with Schauder's theorem give a fixed point instead.
Either way the fixed point solves \eqref{eq:SD_dynamics}.

Viability is therefore not an extra condition imposed on the fixed point. It is
forced by \eqref{eq:varconst}, which writes $x(t)$ as a convex combination of
$x_0$ and values of $T$, all of which lie in $\setS$. No tangency condition on the
right-hand side is needed.
\end{proof}

\begin{remark}\label{rem:existence_known}
For $\setS$ closed convex in a Hilbert space, $\matM\equiv I$ and
$F=\nabla\varphi$, Bolte obtains existence and uniqueness from the
Cauchy--Lipschitz theorem and derives viability from the same
variation-of-constants representation used here \cite[Theorem 2.1]{Bolte2003}.
In the present setting, Assumptions~\ref{ass:metric_bounds} and
\ref{ass:metric_regular} are stated only on $\setS$, so the fixed-point argument
above supplies existence as well as invariance. The invariance step itself uses
only that $T$ maps $\setS$ into $\setS$.
\end{remark}

\begin{lemma}[Uniqueness and forward completeness]\label{lem:unique}
Let Assumptions~\ref{ass:metric_bounds} and~\ref{ass:metric_regular} hold, and let
$F$ be locally Lipschitz on $\setS$. Then the solution of \eqref{eq:SD_dynamics}
from any $x_0\in\setS$ is unique. If in addition
Assumption~\ref{ass:F_Lipschitz} holds, that solution exists for all $t\ge0$.
\end{lemma}

\begin{proof}
Let $x,y$ be two solutions on a common compact interval, and let
$\mathcal K\subset\setS$ be a compact set containing both ranges.
Assumption~\ref{ass:metric_regular} gives a Lipschitz constant for $\matM$ on
$\mathcal K$, so \eqref{eq:inv_metric_var} holds on $\mathcal K$ with that constant
in place of $K_M$. The anchor $\|T(z)-z\|$ is bounded on $\mathcal K$ by
\eqref{eq:displacement}, since $F$ is continuous there. Hence
Proposition~\ref{prop:proj_joint} and \eqref{eq:proj_L} make $T$ Lipschitz on
$\mathcal K$ with some constant $\ell$, and the right-hand side of
\eqref{eq:SD_dynamics} is
Lipschitz with constant $\lambda(\ell+1)$. Gr\"onwall's inequality applied to
$t\mapsto\|x(t)-y(t)\|$ gives $x\equiv y$. Convexity of $\setS$ enters only through
the projection estimates of Section~\ref{sec:projreg}. The Gr\"onwall step itself
does not use it.

For completeness, note that local Lipschitz continuity alone does not exclude
finite-time escape. Under Assumption~\ref{ass:F_Lipschitz}, however, $F$ has linear
growth, $\|F(x)\|\le\|F(x_r)\|+K\|x-x_r\|$ for any fixed $x_r\in\setS$, so
\eqref{eq:displacement} gives $\|\dot x\|\le\lambda\alpha L\|F(x)\|\le a+b\|x\|$
for constants $a,b$ depending only on $\lambda,\alpha,L,K$ and $x_r$. Suppose the
maximal solution had $t_{\max}<\infty$. Gr\"onwall then bounds $\|x(t)\|$ on
$[0,t_{\max})$, so $\|\dot x\|$ is bounded there as well and $x$ is Lipschitz.
Hence $x^*:=\lim_{t\uparrow t_{\max}}x(t)$ exists, and $x^*\in\setS$ because the
solution is $\setS$-valued and $\setS$ is closed. Lemma~\ref{lem:existence} started
at $x^*$ then continues the solution past $t_{\max}$, contradicting maximality. So
$t_{\max}=+\infty$. The usual continuation alternative is not invoked, since it is
stated for fields defined on an open set.
\end{proof}

We now show that the equilibria of \eqref{eq:SD_dynamics}
coincide with solutions of the variational inequality
$\mathrm{VI}(F,\setS)$.

\begin{proposition}[Equilibrium--solution equivalence]\label{prop:equilibrium_VI}
A point $\bar x \in \setS$ is an equilibrium of
\eqref{eq:SD_dynamics} if and only if
$\bar x$ solves $\mathrm{VI}(F,\setS)$.
\end{proposition}

\begin{proof}
By definition, $\bar x$ is an equilibrium if and only if
\[
\bar x
=
P_{\setS,\matM(\bar x)^{-1}}
\bigl(\bar x - \alpha \matM(\bar x)F(\bar x)\bigr).
\]
Using the variational characterization
\eqref{eq:proj_char} with
$y = \bar x - \alpha \matM(\bar x)F(\bar x)$,
this condition is equivalent to
\[
\alpha\inner{F(\bar x)}{x - \bar x} \ge 0,
\quad \forall x \in \setS,
\]
which coincides with \eqref{eq:VI}.
\end{proof}

Having established that equilibria of \eqref{eq:SD_dynamics} coincide with VI
solutions, we now address whether trajectories approach them.

\section{Stability analysis}\label{sec:stability}

We first derive consequences under anchored nonexpansiveness, then treat constant
and state-dependent metrics separately. Only the state-dependent case introduces
variation of $\matM(\cdot)$ into the estimates.

\subsection{Stability under a nonexpansiveness hypothesis}\label{sec:stabI}

The energy calculation below separates what follows from the standing assumptions
from the additional condition needed for stability. Fix
$\bar x\in\operatorname{SOL}(F,\setS)$ and use the inverse metric frozen at $\bar x$:
\begin{equation}\label{eq:Hbar}
\bar Q:=\matM(\bar x)^{-1},
\qquad
V(x):=\tfrac12\|x-\bar x\|_{\bar Q}^2 .
\end{equation}
The choice $\bar Q$, rather than $\matM(\bar x)$, matches the projection geometry
in \eqref{eq:Tmap}. The following identity is assumption-free.

\begin{lemma}[Energy identity]\label{lem:energy_identity}
Let $H\succ0$ be fixed, let $\bar x\in\R^n$, and put
$V_H(x):=\tfrac12\|x-\bar x\|_H^2$. Along any solution of
\eqref{eq:SD_dynamics},
\begin{equation}\label{eq:energy_identity}
\frac{d}{dt}V_H(x(t))
=
\frac{\lambda}{2}
\Bigl(
\|T(x)-\bar x\|_H^2-\|x-\bar x\|_H^2-\|T(x)-x\|_H^2
\Bigr).
\end{equation}
\end{lemma}

\begin{proof}
Write $a:=x-\bar x$ and $b:=T(x)-\bar x$, so that $b-a=T(x)-x$. Then
$\frac{d}{dt}V_H=\lambda\inner{a}{b-a}_H$, and expanding
$\|b\|_H^2=\|a+(b-a)\|_H^2$ gives
$2\inner{a}{b-a}_H=\|b\|_H^2-\|a\|_H^2-\|b-a\|_H^2$.
\end{proof}

The identity shows that stability requires a sign on its first two terms. We
therefore impose the following local condition on $(T,\bar x)$, not on the data.

\begin{assumption}[Anchored nonexpansiveness]\label{ass:T_nonexp}
There exists a neighborhood $U\subset\setS$ of $\bar x$ relative to $\setS$,
meaning that $\{x\in\setS:\|x-\bar x\|<\delta\}\subset U$ for some $\delta>0$, such that
\begin{equation}\label{eq:T_nonexp}
\|T(x)-\bar x\|_{\bar Q}\le\|x-\bar x\|_{\bar Q},
\qquad x\in U.
\end{equation}
\end{assumption}

By Lemma~\ref{lem:energy_identity}, condition \eqref{eq:T_nonexp} is pointwise
equivalent to
\[
\frac{d}{dt}V(x)\le-\frac{\lambda}{2}\|T(x)-x\|_{\bar Q}^2.
\]
Thus the result below is conditional: anchored nonexpansiveness is the required
dissipation estimate, not a regularity consequence of $\matM$ or $F$.

\begin{theorem}[Lyapunov consequences of anchored nonexpansiveness]
\label{thm:Lyap_stability_I}
Let Assumptions~\ref{ass:metric_bounds}, \ref{ass:metric_regular} and
\ref{ass:F_Lipschitz} hold, let $\bar x\in\operatorname{SOL}(F,\setS)$, and suppose
Assumption~\ref{ass:T_nonexp} holds on $U$. Every trajectory remaining in $U$
satisfies
\begin{equation}\label{eq:Vdot_bound}
\frac{d}{dt}V(x(t))\le-\frac{\lambda}{2}\|T(x(t))-x(t)\|_{\bar Q}^2\le0 .
\end{equation}
Hence $\bar x$ is Lyapunov stable relative to $\setS$: if $\varepsilon>0$ and
$\setS\cap\overline B(\bar x,\varepsilon)\subset U$, then every solution with
$\|x(0)-\bar x\|<\varepsilon/L$ is global and remains in
$\overline B(\bar x,\varepsilon)$. The set
$\Omega:=\{x\in\setS:\|x-\bar x\|_{\bar Q}\le\varepsilon/\sqrt L\}$ is compact.
It is positively invariant, meaning that every trajectory starting in $\Omega$
remains in $\Omega$ for all $t\ge0$, and each such trajectory satisfies
\[
\int_0^\infty\|T(x(t))-x(t)\|_{\bar Q}^2\,dt
\le\frac1\lambda\|x(0)-\bar x\|_{\bar Q}^2,
\qquad \|T(x(t))-x(t)\|\to0.
\]
Its $\omega$-limit set, the set of all limits of $x(t_j)$ along sequences
$t_j\to\infty$, is a nonempty compact connected subset of
$\operatorname{SOL}(F,\setS)\cap\Omega$; consequently
$\operatorname{dist}\bigl(x(t),\operatorname{SOL}(F,\setS)\bigr)\to0$.
\end{theorem}

\begin{proof}
Lemma~\ref{lem:energy_identity} and \eqref{eq:T_nonexp} give \eqref{eq:Vdot_bound}.
Put $\rho_\varepsilon:=\varepsilon/\sqrt L$ and define $\Omega$ as above. The metric
bounds give $\Omega\subset\overline B(\bar x,\varepsilon)\subset U$, and
$\|x(0)-\bar x\|<\varepsilon/L$ implies
$\|x(0)-\bar x\|_{\bar Q}<\rho_\varepsilon$. Since $V$ is nonincreasing in $U$, a
first-exit argument makes $\Omega$ positively invariant. Compactness of $\Omega$
and the continuation step closing the proof of Lemma~\ref{lem:unique} then give
global existence. Integrating \eqref{eq:Vdot_bound} gives the stated bound. On
$\Omega$ the vector field is bounded and $T$ is Lipschitz by
Proposition~\ref{prop:proj_joint}, hence the squared residual is uniformly
continuous, and Barbalat's lemma gives its convergence to zero.

For $s\ge0$, let $K_s:=\overline{\{x(t):t\ge s\}}$. Each $K_s$ is nonempty,
compact and connected, and these sets are nested. Their intersection is the
$\omega$-limit set, so it is nonempty, compact and connected. Autonomy and
uniqueness of the flow make it invariant. Since $V(x(t))$ is nonincreasing,
continuity makes $V$ constant on that limit set. Equation~\eqref{eq:Vdot_bound} forces
$T(x)=x$ at every point of that limit set, and
Proposition~\ref{prop:equilibrium_VI} identifies those points as solutions of
$\mathrm{VI}(F,\setS)$. Because the trajectory remains in
compact $\Omega$ and all its subsequential limits lie in
$\operatorname{SOL}(F,\setS)$, the distance conclusion follows.
\end{proof}

The conclusion is attraction to the local solution set, not point convergence. If
$\|T(x)-z\|_{\bar Q}\le\|x-z\|_{\bar Q}$ for every
$z\in\operatorname{SOL}(F,\setS)\cap U$, then the energy identity makes
$t\mapsto\|x(t)-z\|_{\bar Q}$ nonincreasing. Any accumulation point $z_*$ is a
solution by the theorem. Along a subsequence the distance to $z_*$ tends to zero,
so its monotonicity forces the whole trajectory to converge to $z_*$. The limit
depends on $x(0)$.

Monotonicity does not supply \eqref{eq:T_nonexp}, even when the metric is fixed.
Take $\setS=\R^2$, $\matM\equiv I$, $F(x)=Jx$, and $\bar x=0$, where
$J=\begin{psmallmatrix}0&-1\\1&0\end{psmallmatrix}$. Since $J$ is skew, $F$ is
monotone, and $\bar x$ is the unique solution. The projection is the identity, so
$T(x)=A_\alpha x$ with $A_\alpha=I-\alpha J$ and
$A_\alpha^\top A_\alpha=(1+\alpha^2)I$. For every $H\succ0$,
\[
\det\bigl(A_\alpha^\top HA_\alpha\bigr)=(1+\alpha^2)^2\det H>\det H ,
\]
whereas $X\preceq Y$ for positive definite $X,Y$ implies $\det X\le\det Y$. Thus
$A_\alpha^\top HA_\alpha\preceq H$ is impossible, and \eqref{eq:T_nonexp} fails for
every $\alpha>0$ and every fixed $H\succ0$. The flow $\dot x=-\lambda\alpha Jx$ has
circular trajectories, so $0$ is stable but not attracting.

This example has a different role from the other two uses of $J$.
Example~\ref{ex:mechanism} shows pointwise cancellation along two lines by a fixed
non-Euclidean matrix, whereas Section~\ref{sec:counterexample} rotates that matrix
to create an annulus under strong monotonicity. Here $\matM\equiv I$, and the
obstruction is nonexpansiveness in every fixed metric. Cocoercivity, treated in
Section~\ref{sec:stabII}, does suffice.

Related fixed-metric failures are known. Bolte observes that the projected-gradient
residual need not be monotone and replaces the quadratic energy by one augmented
with the potential \cite{Bolte2003}. Hu and Wang exhibit periodic trajectories for
an affine monotone operator on a box in $\R^3$, refuting a proposed convergence
result under pseudomonotonicity \cite{Hu2006}. Bolte's substitute is unavailable
for a general operator without a potential, and the positive fixed-metric results
add structure beyond monotonicity, typically symmetry of $\nabla F$
\cite{Xia2004,Hu2006}.

\subsection{Convergence and rates at a fixed metric}\label{sec:stabII}

Throughout this part the metric is constant, $\matM(x)\equiv M\succ0$, so that
$K_M=0$ is admissible in Assumption~\ref{ass:M_Lipschitz}. Write
$H_M:=(\R^n,\inner{\cdot}{\cdot}_{M^{-1}})$, and put
\begin{equation}\label{eq:Gmap}
G(x):=x-\alpha MF(x),
\qquad
T=P_{\setS,M^{-1}}\circ G .
\end{equation}
The whole section rests on one algebraic cancellation. For $h:=x-y$ and
$\Delta F:=F(x)-F(y)$,
\begin{equation}\label{eq:expansion}
\|G(x)-G(y)\|_{M^{-1}}^2
=
\|h\|_{M^{-1}}^2
-2\alpha\inner{\Delta F}{h}
+\alpha^2\,\Delta F^\top M\,\Delta F .
\end{equation}
The cross term is Euclidean and the quadratic term is $M$-weighted, because
$M^{-1}M=I$ removes one factor from each. This is a property of pairing the
projection metric with the inverse of the preconditioner, and it fails if the two
are chosen independently.

A weaker hypothesis than strong monotonicity also suffices at a fixed metric. Call
$F$ \emph{$M$-cocoercive with modulus $\sigma_M>0$} if
$\inner{F(x)-F(y)}{x-y}\ge\sigma_M(F(x)-F(y))^\top M(F(x)-F(y))$ for all
$x,y\in\setS$. Abbas and Attouch study
$\dot x+x-\operatorname{prox}_{\alpha\Phi}(x-\alpha\mathcal G(x))=0$ for $\Phi$ convex
lower semicontinuous and $\mathcal G$ cocoercive with modulus $\beta$, and obtain weak
convergence to a zero of $\partial\Phi+\mathcal G$ for every $\alpha\in(0,4\beta)$
\cite{Abbas2015}. Taking $\Phi$ to be the indicator of $\setS$ recovers the
cocoercive case of the present flow at $\matM\equiv I$, whose convergence
therefore follows from \cite{Abbas2015}. The change of variables below carries the
general constant metric to that case.

The general constant metric reduces to the Euclidean case. Setting
\begin{equation}\label{eq:change_of_var}
y:=M^{-1/2}x,
\qquad
C:=M^{-1/2}\setS,
\qquad
\widetilde F(y):=M^{1/2}F(M^{1/2}y)
\end{equation}
turns $\|\cdot\|_{M^{-1}}$ into the Euclidean norm and $P_{\setS,M^{-1}}$ into
$P_C$, and the flow into
$\dot y=\lambda\bigl(P_C(y-\alpha\widetilde F(y))-y\bigr)$. Since
$\inner{\widetilde F(y_1)-\widetilde F(y_2)}{y_1-y_2}=\inner{\Delta F}{h}$ and
$\|\widetilde F(y_1)-\widetilde F(y_2)\|^2=\Delta F^\top M\Delta F$, the map
$\widetilde F$ is cocoercive with modulus $\sigma_M$ exactly when $F$ is
$M$-cocoercive with that modulus.

If $F$ is strongly monotone and Lipschitz then cocoercivity is automatic. Under
Assumptions~\ref{ass:metric_bounds}, \ref{ass:F_Lipschitz}
and~\ref{ass:strong_monotone} one has $\sigma_M\ge\mu/(LK^2)$, since
$\Delta F^\top M\Delta F\le LK^2\|h\|^2$ and $\inner{\Delta F}{h}\ge\mu\|h\|^2$.
The state-dependent analysis below runs on strong monotonicity throughout, so the
cocoercive branch is not pursued further.

Strong monotonicity gives more than convergence. It gives an explicit rate, and
an explicit window of step sizes on which that rate holds.

\begin{theorem}[Fixed-metric exponential rate]\label{thm:fixed_metric_rate}
Let $\setS$ be nonempty closed convex, let $M\succ0$ be fixed and satisfy
Assumption~\ref{ass:metric_bounds}, and let Assumptions~\ref{ass:F_Lipschitz} and
\ref{ass:strong_monotone} hold. If
\begin{equation}\label{eq:window}
0<\alpha<\frac{2\mu}{LK^2},
\qquad
\rho:=\sqrt{1-\frac{2\alpha\mu}{L}+\alpha^2K^2},
\end{equation}
then $\rho\in[0,1)$, the map $T$ is a $\rho$-contraction of $\setS$ in $H_M$, it
has a unique fixed point $\bar x$, which is the unique solution of
$\mathrm{VI}(F,\setS)$, and every trajectory satisfies
\begin{equation}\label{eq:fixed_rate}
\|x(t)-\bar x\|_{M^{-1}}\le e^{-\lambda(1-\rho)t}\|x(0)-\bar x\|_{M^{-1}},
\qquad
\|x(t)-\bar x\|\le\sqrt{\kappa_2(M)}\;e^{-\lambda(1-\rho)t}\|x(0)-\bar x\| .
\end{equation}
Here $\kappa_2(M)=\lambda_{\max}(M)/\lambda_{\min}(M)$ is the spectral condition
number of $M$, so $\kappa_2(M)\le L^2$ under Assumption~\ref{ass:metric_bounds}.
The choice $\alpha_*=\mu/(LK^2)$ minimizes $\rho$, with
$\rho_*=\sqrt{1-(\mu/(LK))^2}$.
\end{theorem}

\begin{proof}
From \eqref{eq:expansion}, Assumption~\ref{ass:strong_monotone} and
Assumption~\ref{ass:F_Lipschitz},
\[
\|G(x)-G(y)\|_{M^{-1}}^2
\le\|h\|_{M^{-1}}^2+\bigl(-2\alpha\mu+\alpha^2LK^2\bigr)\|h\|^2 .
\]
The bracket is negative precisely when $\alpha<2\mu/(LK^2)$. Only then may
$\|h\|^2\ge L^{-1}\|h\|_{M^{-1}}^2$ be substituted, which yields
$\rho^2$ as displayed. It is essential that the two perturbation terms be combined
in the Euclidean norm before this substitution: converting them separately bounds
the positive term by $\alpha^2L^2K^2\|h\|^2_{M^{-1}}$ and yields the strictly
weaker window $\alpha<2\mu/(L^3K^2)$. Since
Assumptions~\ref{ass:F_Lipschitz} and \ref{ass:strong_monotone} give $\mu\le K$
and $L\ge1$, the radicand is nonnegative and $\rho$ is real.

Nonexpansiveness of $P_{\setS,M^{-1}}$ in $H_M$ gives the contraction. As $\setS$
is closed and $\|\cdot\|_{M^{-1}}$ is equivalent to the Euclidean norm, $\setS$ is
complete in $H_M$ and Banach's theorem applies. Finally
Lemma~\ref{lem:energy_identity} with $H=M^{-1}$ gives
$\frac{d}{dt}\|x-\bar x\|_{M^{-1}}\le\lambda(\rho-1)\|x-\bar x\|_{M^{-1}}$, and
Gr\"onwall yields the weighted estimate; the Euclidean form follows from
$\lambda_{\max}(M^{-1})$ and $\lambda_{\min}(M^{-1})$.
\end{proof}

\begin{remark}\label{rem:fixed_metric_known}
Fixed-metric exponential convergence follows from existing results. For
$F=\nabla f$ with $f$
strongly convex and $\nabla f$ Lipschitz, write $\mu_f$ and $K_f$ for their
respective moduli. Antipin gives the Euclidean flow an explicit exponential rate
and optimizes the step, obtaining $\alpha_{\mathrm{opt}}=4/(K_f+\mu_f)$ with rate
$4K_f\mu_f/(K_f+\mu_f)^2$ \cite[Theorem 3]{Antipin1994}. For a general operator,
taking the two operators to be the normal cone of $\setS$ and $F$ itself, the flow
\eqref{eq:SD_dynamics} at $\matM\equiv I$ is the
forward--backward dynamic of \cite{Bot2018}, whose Theorem~1 gives global
exponential convergence under strong monotonicity, and the change of variables
\eqref{eq:change_of_var} carries the general constant metric to that case, with
$\widetilde F$ still strongly monotone and Lipschitz. We use
\eqref{eq:window} below as an explicit baseline in $(\mu,K,L,\alpha)$. It is
attained on some instances. For $\setS=\R$, $\matM\equiv1$ and
$F(x)=kx$ one has $\mu=K=k$, and for $0<\alpha\le1/k$ the flow is
$\dot x=-\lambda\alpha kx$ with $\rho=1-\alpha k$, so the displayed rate
$\lambda(1-\rho)=\lambda\alpha k$ is attained.

The certificate does reduce correctly in the limit that connects the two
continuous-time models. Take $\matM\equiv I$, so $L=1$, and let $\alpha\downarrow0$
with $\lambda=1/\alpha$, which is the regime in which the residual field converges
to the tangent-cone field of the classical projected dynamical system. Then
$\rho=\sqrt{1-2\alpha\mu+\alpha^2K^2}=1-\alpha\mu+O(\alpha^2)$, so the rate
$\lambda(1-\rho)$ tends to $\mu$. That is exactly the exponential rate obtained for
the projected dynamical system under strong monotonicity with modulus $\mu$,
established for that flow both by Lyapunov argument \cite{Nagurney1996} and by
contraction analysis \cite{Allibhoy2025}. The step-size window and the constants in \eqref{eq:window}
are therefore artefacts of working at a fixed $\alpha>0$, and not a weaker
conclusion.
\end{remark}

\subsection{Convergence and rates at a state-dependent metric}\label{sec:stabIII}

The reduction in Remark~\ref{rem:fixed_metric_known} requires $M$ constant. When
the metric genuinely varies the flow is not a Euclidean one in disguise, and a
different argument is needed. The one below uses the metric at the current state
rather than a frozen one, which is what removes the perturbation terms that a
frozen comparison would produce.

Throughout this part Assumptions~\ref{ass:metric_bounds}--\ref{ass:M_Lipschitz}
and \ref{ass:strong_monotone} are in force, $\alpha$ and $\rho$ are as in
\eqref{eq:window}, and we write
\begin{equation}\label{eq:current_metric_V}
Q(x):=\matM(x)^{-1},
\qquad
h:=x-\bar x,
\qquad
r(x):=\|h\|_{Q(x)},
\qquad
V(x):=\tfrac12 r(x)^2 .
\end{equation}

The starting point is an identity that holds in every metric, and that is what
makes the current-metric argument possible.

\begin{lemma}[Solution characterization in every metric]\label{lem:cancellation}
Let $\bar x\in\operatorname{SOL}(F,\setS)$. Then for \emph{every} $x$ at which
$\matM(x)$ is defined,
\begin{equation}\label{eq:cancellation}
\bar x=P_{\setS,\matM(x)^{-1}}\bigl(\bar x-\alpha \matM(x)F(\bar x)\bigr).
\end{equation}
\end{lemma}

\begin{proof}
Put $Q=\matM(x)^{-1}$ and $y=\bar x-\alpha\matM(x)F(\bar x)$. Then
$Q(\bar x-y)=\alpha Q\matM(x)F(\bar x)=\alpha F(\bar x)$, so for every
$w\in\setS$,
$\inner{Q(\bar x-y)}{w-\bar x}=\alpha\inner{F(\bar x)}{w-\bar x}\ge0$
because $\bar x$ solves the variational inequality. By \eqref{eq:proj_char} this
is exactly \eqref{eq:cancellation}.
\end{proof}

Identity \eqref{eq:cancellation} uses only symmetry and invertibility of
$\matM(x)$, holds for every $\alpha>0$, and does not require $x\in\setS$. It is
the reason the state-dependent case is tractable, because it lets $T(x)$ and
$\bar x$ be compared as two images of the \emph{same} projection
$P_{\setS,Q(x)}$, so no metric-mismatch term arises.

Bonettini, Zanella and Zanni establish the corresponding stationarity identity
for every step size and every fixed positive definite matrix in the gradient
setting \cite[Lemma 2.2]{Bonettini2009}. Substituting $D=\matM(x)$ yields
\eqref{eq:cancellation}; the form above emphasizes that $x$ varies while
$\bar x$ remains fixed.

\begin{lemma}[Contraction in the current metric]\label{lem:current_contraction}
For every $x\in\setS$,
\begin{equation}\label{eq:current_contraction}
\|T(x)-\bar x\|_{Q(x)}\le\rho\,\|x-\bar x\|_{Q(x)} .
\end{equation}
\end{lemma}

\begin{proof}
By Lemma~\ref{lem:cancellation} both $T(x)$ and $\bar x$ are images under
$P_{\setS,Q(x)}$, of $x-\alpha\matM(x)F(x)$ and
$\bar x-\alpha\matM(x)F(\bar x)$ respectively. Lemma~\ref{lem:proj_arg_var} gives
$\|T(x)-\bar x\|_{Q(x)}\le\|h-\alpha\matM(x)\bigl(F(x)-F(\bar x)\bigr)\|_{Q(x)}$,
and the right-hand side is estimated exactly as in the proof of
Theorem~\ref{thm:fixed_metric_rate}, with $M$ replaced by $\matM(x)$ throughout.
No term involving $F(\bar x)$ or $K_M$ appears.
\end{proof}

The metric-derivative term is what confines the next theorem to a neighbourhood.
In the
non-Euclidean flows of convex programming the geometry is generated by a convex
potential. A Legendre function $\varphi$ induces the Hessian metric $H=\nabla^2\varphi$, and
the Lyapunov functional is then the Bregman distance $D_\varphi(\bar x,\cdot)$, whose
gradient in that metric is the displacement field $x-\bar x$ exactly
\cite{Alvarez2004,Attouch2004}. No derivative of the metric survives. Alvarez,
Bolte and Brahic show that this is available for Hessian metrics and for no
others. Their Theorem~3.1 states that the fields $x\mapsto x-y$ are gradients in
the metric $H$ precisely when $H=\nabla^2\varphi$ for a strictly convex $\varphi$ of class
$C^3$, which forces the integrability conditions
$\partial_iH_{jk}=\partial_jH_{ik}$ \cite{Alvarez2004}. Mirror descent dynamics for monotone variational inequalities
sit in the same place, with the Fenchel coupling of a distance-generating function
in the role of the Bregman distance \cite{Mertikopoulos2018}.

The class considered here permits arbitrary Lipschitz positive definite
matrix-valued maps and therefore does not generally supply the canonical Bregman
functional used in the cited proofs. We instead use the weighted quadratic $V$ of
\eqref{eq:current_metric_V}; its derivative contains a term in $\dot{\matM}$ of
indefinite sign. Controlling that term yields an exponential rate on a
neighbourhood of $\bar x$ where metric variation cannot overturn the pointwise
contraction.

\begin{theorem}[State-dependent local exponential stability]\label{thm:exponential}
Let $K_M>0$ and set
\begin{equation}\label{eq:CL}
C_L:=\tfrac12L^{3/2}K_M(1+\rho).
\end{equation}
Choose $R$ with $0<R<(1-\rho)/C_L$ and put $\eta_R:=(1-\rho)-C_LR>0$. Then the
sublevel set $\{x\in\setS: r(x)\le R\}$ is positively invariant, every trajectory
with $r(x(0))\le R$ exists for all $t\ge0$, and
\begin{equation}\label{eq:local_rate}
r(x(t))\le e^{-\lambda\eta_R t}\,r(x(0)),
\qquad
\|x(t)-\bar x\|\le L\,e^{-\lambda\eta_R t}\|x(0)-\bar x\| .
\end{equation}
The condition $\|x(0)-\bar x\|<R/\sqrt L$ is sufficient for $r(x(0))<R$. The
choice $R=(1-\rho)/\bigl(L^{3/2}K_M(1+\rho)\bigr)$, half the admissible supremum,
gives $\eta_R=(1-\rho)/2$ and the Euclidean condition
$\|x(0)-\bar x\|<(1-\rho)/\bigl(L^{2}K_M(1+\rho)\bigr)$. We refer to it below as the
half-rate choice, since $\eta_R$ tends to $1-\rho$ as $R\to0$ and to zero as $R$
approaches the supremum.
\end{theorem}

\begin{proof}
Along a trajectory, $\matM(x(t))$ is Lipschitz in $t$ by
Assumption~\ref{ass:M_Lipschitz}, hence so is $Q(x(t))$ by
\eqref{eq:inv_metric_var}, and $V$ is locally absolutely continuous. At almost
every $t$,
\[
\dot V=h^\top Q\dot h+\tfrac12h^\top\dot Qh ,
\qquad
\dot Q=-Q\dot{\matM}Q,
\qquad
\|\dot{\matM}\|\le K_M\|\dot x\| .
\]
For the first term, Cauchy--Schwarz in $\inner{\cdot}{\cdot}_{Q}$ and
\eqref{eq:current_contraction} give
$h^\top Q\dot h=\lambda\bigl(\inner{h}{T(x)-\bar x}_{Q}-r^2\bigr)\le-\lambda(1-\rho)r^2$.
For the second, $Q^2\preceq LQ$ holds because every eigenvalue of $Q$ lies in
$[L^{-1},L]$, so
\[
|h^\top\dot Qh|=\bigl|(Qh)^\top\dot{\matM}(Qh)\bigr|
\le\|\dot{\matM}\|\,h^\top Q^2h\le LK_M\|\dot x\|\,r^2 .
\]
Finally $\|\dot x\|=\lambda\|T(x)-x\|\le\lambda\sqrt L\bigl(\|T(x)-\bar x\|_{Q}+\|h\|_{Q}\bigr)
\le\lambda\sqrt L(1+\rho)r$. Combining,
\begin{equation}\label{eq:Vdot_state}
\dot V\le-\lambda\bigl[(1-\rho)-C_Lr\bigr]r^2 .
\end{equation}
On $\{r\le R\}$ the bracket is at least $\eta_R>0$, so $\dot V\le-2\lambda\eta_RV$.
A first-exit argument identical to that of Theorem~\ref{thm:Lyap_stability_I} makes
the sublevel positively invariant, compactness gives existence for all $t\ge0$, and
Gr\"onwall gives \eqref{eq:local_rate}. The Euclidean statements follow from
$\|h\|\le\sqrt L\,r$ at time $t$ and $r\le\sqrt L\,\|h\|$ at time $0$, the two
conversions using the two halves of Assumption~\ref{ass:metric_bounds}.
\end{proof}

\begin{remark}\label{rem:exp_scope}
If $K_M=0$ then $C_L=0$, the radius restriction disappears, and
Theorem~\ref{thm:exponential} reduces to the global statement of
Theorem~\ref{thm:fixed_metric_rate}, with the Euclidean prefactor sharpening from
$L$ to $\sqrt{\kappa_2(M)}$. For $K_M>0$ the half-rate construction reduces the
certified rate by a factor of two to control metric variation; this comparison
concerns certified bounds rather than actual convergence speeds.
\end{remark}

The standing assumptions do not support a global conclusion. Two sufficient
variation restrictions are recorded here, followed by an example showing that the
standing assumptions alone are insufficient.

\begin{theorem}[Global convergence under variation control]\label{thm:global}
Suppose $\setS$ is compact and put $F_{\max}:=\sup_{x\in\setS}\|F(x)\|$. If
\begin{equation}\label{eq:cond19}
c_g:=1-\rho-\frac{\alpha}{2}L^2K_MF_{\max}>0,
\end{equation}
then every trajectory satisfies $r(x(t))\le e^{-\lambda c_g t}r(x(0))$ and
$\|x(t)-\bar x\|\le L\,e^{-\lambda c_gt}\|x(0)-\bar x\|$. Moreover
\eqref{eq:cond19} holds for some $\alpha>0$ if and only if
$\mu/L>\tfrac12L^2K_MF_{\max}$.
\end{theorem}

\begin{proof}
Lemma~\ref{lem:displacement} gives $\|T(x)-x\|\le\alpha LF_{\max}$, so
$\|\dot x\|\le\lambda\alpha LF_{\max}$. Substituting this in place of the local
velocity bound in the proof of Theorem~\ref{thm:exponential} replaces $C_Lr$ by
$\tfrac{\alpha}{2}L^2K_MF_{\max}$ and gives $\dot V\le-2\lambda c_gV$. For the
characterization, write $a:=\mu/L$ and $c:=\tfrac12L^2K_MF_{\max}$. Since
$\rho^2-(1-a\alpha)^2=\alpha^2(K^2-a^2)\ge0$ we have $1-\rho\le a\alpha$, so
$1-\rho>c\alpha$ is impossible when $c\ge a$; conversely if $a>c$ then
$1-\rho=\alpha(2a-\alpha K^2)/(1+\rho)>\alpha(a-\alpha K^2/2)>c\alpha$ for
$0<\alpha<2(a-c)/K^2$.
\end{proof}

\begin{remark}\label{rem:cond21}
A pathwise refinement is available. Let
$\vartheta(t):=\lambda_{\max}^{+}\bigl(Q^{-1/2}\dot QQ^{-1/2}\bigr)$, the positive
part of the largest eigenvalue of $-\matM^{-1/2}\dot{\matM}\matM^{-1/2}$. The same
computation gives $\dot V\le\bigl[-2\lambda(1-\rho)+\vartheta(t)\bigr]V$, so if
$\int_0^t\vartheta\le2\lambda(1-\rho-\eta)t+\Gamma_0$ for some $\eta\in(0,1-\rho]$ and
$\Gamma_0<\infty$, then $r(x(t))\le e^{\Gamma_0/2}e^{-\lambda\eta t}r(x(0))$. In particular finite
total positive variation $\int_0^\infty\vartheta<\infty$ recovers the fixed-metric
rate with a finite prefactor. On compact $\setS$, \eqref{eq:cond19} implies this
condition with $\Gamma_0=0$. The converse fails, so \eqref{eq:cond19} is the more
restrictive of the two. For a witness, take the metric of
Example~\ref{rem:periodic} and the stationary trajectory $x(t)\equiv\bar x$. Then
$\vartheta\equiv0$, so the pathwise condition holds, while \eqref{eq:cond19} fails
for that metric.

A budget of this kind is what one should expect. The other way of making a
projected flow state dependent is to move the feasible set rather than the
metric, and there the same thing happens. For $C(x)=c(x)+C_0$ with $c$ Lipschitz
with constant $\ell_C$, Antipin, Jaćimović and Mijajlović prove convergence of a
second-order projected method under
\[
\ell_C<\min\Bigl\{\frac{\mu}{K},\ \frac{2\mu}{K^2}\Bigl(K+\mu-\sqrt{\mu(2K+\mu)}\Bigr)\Bigr\}
\]
\cite[Theorem~1]{Antipin2011}. They summarize the restriction by stating that the
feasible set must vary slowly. Condition \eqref{eq:cond19} plays the same role for
a moving metric that their bound plays for a moving set.
\end{remark}

\section{A rotating metric with periodic orbits}\label{sec:counterexample}

We now rotate the eigenframe of the matrix in Example~\ref{ex:mechanism} with the
polar angle.

\begin{example}[A rotating metric]\label{rem:periodic}
Let $\setS=\{x\in\R^2:\|x\|\le1.1\}$ and $F(x)=(I+2J)x$ with
$J=\begin{psmallmatrix}0&-1\\1&0\end{psmallmatrix}$, and put $\lambda=1$ and
$\alpha=0.05$. Then $F$ is $1$-strongly monotone and $\sqrt5$-Lipschitz, and
$\bar x=0$ is the unique solution. For $x\ne0$ set
\[
s=\|x\|,\qquad e_r=x/s,\qquad e_\theta=Je_r,\qquad R(\theta)=[e_r,e_\theta],
\]
and put
\[
B=\begin{psmallmatrix}1&-1/2\\-1/2&1/2\end{psmallmatrix}\succ0,
\qquad
p(u)=10u^3-15u^4+6u^5,
\qquad
\chi(s)=
\begin{cases}
p(s/0.8), & s\le0.8,\\
1, & s>0.8 .
\end{cases}
\]
Define
\begin{equation}\label{eq:counterexample_metric}
\matM(x)=(1-\chi)I+\chi\,R(\theta)BR(\theta)^\top,
\qquad \matM(0)=I .
\end{equation}
Then $\matM$ is $C^2$, uniformly positive definite and Lipschitz, with
$D\matM(0)=0$. With $r_\alpha=1.1/\sqrt{1+\alpha^2/4}$, every circle
$0.8\le s\le r_\alpha$ is a periodic orbit. Moreover $\matM^{-1}$ is not the
Hessian of a potential on that annulus.
\end{example}

The three claims are verified in turn.

\paragraph{Regularity.}
For $0\le u\le1$ one has $p'(u)=30u^2(1-u)^2$, so $0\le p(u)\le1$. Moreover
$p(0)=p'(0)=p''(0)=0$ and $p(1)=1$, $p'(1)=p''(1)=0$. Hence $\chi(s)=O(s^3)$ at the
origin, and $\chi$ matches the constant value $1$ with two derivatives at $s=0.8$. A
derivative of order $k$ of the angular factor is $O(s^{-k})$ for $k=1,2$. The
derivatives of orders zero, one and two of $\matM-I$ are therefore $O(s^3)$,
$O(s^2)$ and $O(s)$. They extend continuously to the origin, and $D\matM(0)=0$. So
$\matM$ is $C^2$ and Lipschitz on the compact set $\setS$. Since $0\le\chi\le1$ and
the eigenvalues of $B$ are $(3\pm\sqrt5)/4$, the metric is uniformly positive
definite and one may take $L=3+\sqrt5$. Differentiating \eqref{eq:counterexample_metric}
gives the certified bound $K_M\le3.571$. The step $\alpha=0.05$ lies inside the
window \eqref{eq:window}, since $2\mu/(LK^2)=0.0764$.

\paragraph{An annulus of periodic orbits.}
On $0.8\le s\le1.1$ one has $\chi=1$ and $B(1,2)^\top=(0,\tfrac12)^\top$, hence
$\matM(x)F(x)=\tfrac12Jx$. For the forward point $y=x-\alpha\matM(x)F(x)$,
\[
\|y\|=s\sqrt{1+\alpha^2/4} .
\]
If $0.8\le s\le r_\alpha$ then $y\in\setS$, the projection is inactive, and
$\dot x=-\tfrac{\lambda\alpha}{2}Jx$. Every circle in this range is a periodic
orbit. For $0<s<0.8$ the same computation in the local frame gives
\[
y=s\bigl[(1-\alpha(1-\chi))e_r-\alpha(2-3\chi/2)e_\theta\bigr],
\qquad
\|y\|\le s\sqrt{1+\alpha^2/4}<1.1 ,
\]
so the projection is again inactive and $\dot s=-\lambda\alpha(1-\chi(s))s<0$. The
circle $s=0.8$ is invariant and uniqueness prevents trajectories from crossing it.
Thus the basin of attraction of $\bar x=0$, meaning the set of initial points whose
trajectories converge to $\bar x$, is exactly the open disk $s<0.8$. Compactness of
$\setS$, uniform positive definiteness and Lipschitz continuity of $\matM$, strong
monotonicity, $F(\bar x)=0$ and $D\matM(\bar x)=0$ therefore do not imply global
convergence.

\paragraph{Failure of Hessian integrability.}
This places the example outside the hypotheses of
\cite{Amochkina1997,Mijajlovic2018}. On the annulus
$\matM(x)^{-1}=R(\theta)B^{-1}R(\theta)^\top$ with
$B^{-1}=\begin{psmallmatrix}2&2\\2&4\end{psmallmatrix}$. Write $H:=\matM^{-1}$ and
$H=3I+N(\theta)$. The traceless symmetric matrix $N(\theta)$ has eigenvalues
$\pm\sqrt5$. Direct multiplication gives, with $\psi=2\theta+\psi_0$,
\[
H_{11}=3+\sqrt5\cos\psi,\qquad
H_{12}=\sqrt5\sin\psi,\qquad
H_{22}=3-\sqrt5\cos\psi,
\]
where $\cos\psi_0=-1/\sqrt5$ and $\sin\psi_0=2/\sqrt5$. Since
$\partial_1\theta=-x_2/s^2$ and $\partial_2\theta=x_1/s^2$, the two independent
conditions in the plane give
\begin{equation}\label{eq:counterexample_residual}
\begin{pmatrix}
\partial_1H_{22}-\partial_2H_{12}\\[2pt]
\partial_2H_{11}-\partial_1H_{12}
\end{pmatrix}
=-\frac{2\sqrt5}{s}
\begin{pmatrix}\cos(\theta+\psi_0)\\[2pt] \sin(\theta+\psi_0)\end{pmatrix},
\qquad\text{of norm}\quad \frac{2\sqrt5}{s} .
\end{equation}
Each component vanishes on isolated rays, but the pair never vanishes together. So
the integrability conditions $\partial_iH_{jk}=\partial_jH_{ik}$ fail at
\emph{every} point of the annulus. The residual norm ranges from
$2\sqrt5/r_\alpha\approx4.07$ at the outer periodic radius to
$2\sqrt5/0.8\approx5.59$ at the inner radius. Hence no $C^3$ potential $\varphi$ can
satisfy $\matM^{-1}=\nabla^2\varphi$ on the annulus.

\begin{remark}[Relation to the convergence certificates]\label{rem:periodic_consistency}
The pointwise estimate \eqref{eq:current_contraction} holds throughout the periodic
annulus. Since the projection is inactive there, its exact ratio is
\[
\Bigl(\frac{2-2\alpha+\alpha^2}{2}\Bigr)^{1/2}=0.9753<\rho=0.9967 .
\]
If $s=\|x\|$, $h=x$ and $Q=\matM^{-1}$ along one of these orbits, then
\[
h^\top Q\dot h=-\lambda\alpha s^2,
\qquad
\tfrac12h^\top\dot Qh=+\lambda\alpha s^2 .
\]
Thus $\dot V=0$, so the metric-derivative term in \eqref{eq:Vdot_state} exactly
cancels the actual contraction. The admissible local-radius supremum from the same
certified constants is $7.7364\times10^{-5}$ in the weighted norm, whereas the
periodic annulus begins at weighted radius $0.8\sqrt2$. There is therefore no
contradiction with Theorem~\ref{thm:exponential}.

The example does not show that condition \eqref{eq:cond19} is necessary. It shows
that the standing assumptions alone do not imply global convergence.
\end{remark}

\paragraph{An integrable metric for contrast.}
Take $\setS=[0,1]^n$ and the separable diagonal metric
\begin{equation}\label{eq:separable_metric}
\matM(x)=\diag\bigl(m_1(x_1),\dots,m_n(x_n)\bigr),
\qquad
m_i(t)=m_{\min}+(m_{\max}-m_{\min})\,4t(1-t),
\end{equation}
with $0<m_{\min}\le m_{\max}$. For $t\in[0,1]$ one has $0\le4t(1-t)\le1$, hence
\[
m_{\min}I\preceq\matM(x)\preceq m_{\max}I,
\qquad
\|\matM(x)\|\le m_{\max},
\qquad
\|\matM(x)^{-1}\|\le m_{\min}^{-1} .
\]
Assumption~\ref{ass:metric_bounds} therefore holds with
$L_{\rm sep}:=\max\{m_{\max},m_{\min}^{-1}\}$, and $L_{\rm sep}=5$ for the values
used in Section~\ref{sec:numerics}.

Writing $H:=\matM^{-1}$ gives $H_{jk}(x)=\delta_{jk}/m_k(x_k)$, where $\delta_{jk}$
is the Kronecker delta, so both sides of the integrability condition reduce to
\[
\partial_iH_{jk}=\delta_{jk}\delta_{ik}\,(1/m_k)'(x_k)=\partial_jH_{ik},
\]
each supported on the single index triple $i=j=k$. The conditions hold on $\setS$,
the potential is $\varphi(x)=\sum_i\varphi_i(x_i)$ with $\varphi_i''=1/m_i$, and
$\matM^{-1}=\nabla^2\varphi$ there. Since $0<m_i\le m_{\max}$, each
$\varphi_i''\ge1/m_{\max}>0$, so $\varphi$ is strongly convex on $\setS$ with
modulus $1/m_{\max}$. The polynomial formula is needed only on $[0,1]$. To apply the
cited whole-space results, extend each $m_i$ smoothly to $\widetilde m_i$ on $\R$,
equal to $m_i$ on a neighbourhood of $[0,1]$, with
$0<\underline m\le\widetilde m_i\le\overline m<\infty$. This preserves separability
and gives a globally strongly convex potential.

Within the diagonal class, Hessian integrability is controlled by coordinate
dependence. It does not follow from a fixed eigenframe. Let
\[
H(x)=\diag\bigl(a(x_2),b(x_1)\bigr),
\]
where $a$ and $b$ are positive $C^1$ functions. The two planar conditions give
\[
a'(x_2)=\partial_2H_{11}=\partial_1H_{21}=0,
\qquad
b'(x_1)=\partial_1H_{22}=\partial_2H_{12}=0,
\]
so $H$ is Hessian-integrable only when $a$ and $b$ are constant. The nonzero control
$H(x)=\diag(1+x_2^2/4,\,1)$ used in Section~\ref{sec:numerics} is a concrete case. A
turning eigenframe is not an obstruction by itself either. For $\varepsilon>0$,
\[
\nabla^2\Bigl(\tfrac12\|x\|^2+\tfrac{\varepsilon}{4}\|x\|^4\Bigr)
=(1+\varepsilon\|x\|^2)I+2\varepsilon xx^\top
\]
is a positive definite Hessian whose radial and tangential eigenframe turns with the
polar angle. In Example~\ref{rem:periodic} the failure is proved by
\eqref{eq:counterexample_residual}, not by rotation alone.

\section{Numerical study}\label{sec:numerics}

This section evaluates the metric of Example~\ref{rem:periodic},
Theorem~\ref{thm:exponential} and condition~\eqref{eq:cond19} on an analytically
solvable family. The certificate bounds are proved in
Section~\ref{sec:stability}. The analytic trajectory and residual formulas are
derived in Section~\ref{sec:counterexample}, which also supplies the separable
control \eqref{eq:separable_metric}. The computations reproduce selected trajectories,
check the residual calculation against known controls, and evaluate the stated
certificate bounds. We report no timings and no comparisons with other methods.

\subsection{Instances and protocol}\label{sec:num_setup}

All computations are deterministic and use \texttt{Float64} arithmetic in Julia
1.12.6. Flow trajectories are computed with \texttt{Rodas5P} and finite-difference
Jacobians. The absolute and relative tolerances are $10^{-10}$, the maximum step is
$0.25$, and output is saved at $2001$ fixed times. The certified local trajectory
is repeated at tolerances $10^{-12}$, and that comparison is reported below. At each
right-hand-side evaluation, put $y=x-\alpha\matM(x)F(x)$ and $Q=\matM(x)^{-1}$. If
$\|y\|\le1.1$ then $y\in\setS$, the projection returns $y$ and the multiplier is
$\nu=0$. Otherwise the projection of $y$ in the $Q$-metric is
$(Q+\nu I)^{-1}Qy$ for the unique $\nu>0$ with $\|(Q+\nu I)^{-1}Qy\|=1.1$, obtained
to tolerance $10^{-13}$. The map $\nu\mapsto\|(Q+\nu I)^{-1}Qy\|$ decreases from
$\|y\|$ to $0$, so a nonnegative root exists exactly when $\|y\|\ge1.1$. Writing
$z:=P_{\setS,Q}(y)$, the Karush--Kuhn--Tucker (KKT) primal-feasibility,
dual-feasibility, stationarity and complementarity residuals are, respectively,
\[
\max\{\|z\|-1.1,0\},\qquad
\max\{-\nu,0\},\qquad
\|Q(z-y)+\nu z\|,\qquad
\bigl|\nu(\|z\|^2-1.1^2)\bigr|.
\]
Their maximum is required to stay below $5\times10^{-11}$, and the largest
observed value where $\nu>0$ is reported in
Table~\ref{tab:activeproj}. A failed integration rejects the run. Every run is
recorded with its inputs, its outputs and the exact package versions used to
produce them.

We use the counterexample metric and a one-parameter family containing it. The
first instance uses the metric of Example~\ref{rem:periodic}, with
$\setS=\{x\in\R^2:\|x\|\le1.1\}$, $F(x)=(I+2J)x$, $\lambda=1$ and $\alpha=0.05$. The
interpolation family is
\begin{equation}\label{eq:interp_family}
\matM_\tau(x)=(1-\tau\chi)I+\tau\chi\,R(\theta)BR(\theta)^\top,
\qquad \tau\in[0,1],
\end{equation}
with $\chi$, $B$ and $R(\theta)$ as in Example~\ref{rem:periodic}. The endpoint
$\tau=0$ is the Euclidean metric and $\tau=1$ is the metric of
Example~\ref{rem:periodic}. The regularity argument of
Section~\ref{sec:counterexample} applies for every $\tau\in[0,1]$. The metric is
integrable at $\tau=0$ and non-integrable at every $\tau>0$. Write
$\mathcal{R}_\tau$ for the integrability residual
$(\partial_1H_{22}-\partial_2H_{12},\ \partial_2H_{11}-\partial_1H_{12})$ of
$H=\matM_\tau^{-1}$. On the annulus
$\|\mathcal{R}_\tau(x)\|=\tau\sqrt5/(2\det(D_\tau)\|x\|)$ with
$D_\tau=(1-\tau)I+\tau B$ and $\det D_\tau=1-\tau/2-\tau^2/4$.

Three further metrics test the residual calculation. No trajectories are computed
for them. Set $A=\diag(1,\tfrac12)$ and $c_\varphi=0.35$. The zero controls are the
bounded Hessian metric $\nabla^2\varphi$, with
$\varphi(x)=\tfrac12x^\top Ax+c_\varphi(\sqrt{1+\|x\|^2}-1)$, and the separable diagonal metric
\eqref{eq:separable_metric}. For the latter we use $m_{\min}=1$ and $m_{\max}=5$.
The nonzero control is $H(x)=\diag(1+x_2^2/4,\,1)$, whose residual
is $(0,x_2/2)$.

Table~\ref{tab:constants} collects metric, residual and certificate quantities for
\eqref{eq:interp_family} at four values of $\tau$.

\begin{table}[htbp]
\centering
\caption{Metric, residual and certificate quantities for the family
\eqref{eq:interp_family}. The common
values are $\mu=1$, $K=\sqrt5$, $F_{\max}=1.1\sqrt5$, $\lambda=1$ and $\alpha=0.05$.
The metric bound $L$ is exact, while $K_M^{\rm ub}:=3.571\tau$ is a certified
upper bound. We set
$c_g^{\rm lb}:=1-\rho-\frac{\alpha}{2}L^2K_M^{\rm ub}F_{\max}$. The columns
$\rho$ and $R_{1/2}$ are also derived from those bounds and are not optimal values.
Displayed decimals are rounded evaluations of closed-form
expressions. Here $R_{1/2}$ is the half-rate choice of
Theorem~\ref{thm:exponential}, $\tau_*$ is the root of $c_g^{\rm lb}$, and the
residual is that of \eqref{eq:interp_family} evaluated at $\|x\|=0.95$. The row
$\tau=1$ uses the metric from Example~\ref{rem:periodic}.}
\label{tab:constants}
\begin{tabular}{lcccccc}
\toprule
$\tau$ & $L$ & $K_M^{\rm ub}$ & $\rho$ & $R_{1/2}$ & $c_g^{\rm lb}$ & $\|\mathcal{R}_\tau\|$ at $\|x\|=0.95$ \\
\midrule
$0$      & $1.000000$ & $0$      & $0.955249$ & $+\infty$            & $+0.044751$ & $0$ \\
$0.1$    & $1.088023$ & $0.3571$ & $0.959474$ & $5.1033\cdot10^{-2}$ & $+0.014531$ & $0.124209$ \\
$\tau_*$ & $1.127058$ & $0.4976$ & $0.961131$ & $3.3288\cdot10^{-2}$ & $0$         & $0.177202$ \\
$1$      & $5.236068$ & $3.5710$ & $0.996695$ & $3.8682\cdot10^{-5}$ & $-6.016995$ & $4.707512$ \\
\bottomrule
\end{tabular}
\end{table}

The step $\alpha=0.05$ lies inside the window $2\mu/(LK^2)$ of \eqref{eq:window} for
every $\tau$, because that window is smallest at $\tau=1$, where it equals
$0.076393$. The root is $\tau_*\approx0.1393475122$. At $\tau=0$ one has
$K_M=C_L=0$, so the local inequality $\eta_R=(1-\rho)-C_LR>0$ places no finite
restriction on $R$. The global lower-bound certificate is positive for
$0\le\tau<\tau_*$ and vanishes at $\tau_*$. At $\tau=1$ it gives no convergence
certificate, while Section~\ref{sec:counterexample} supplies periodic solutions.

\subsection{The annulus and the residual}\label{sec:num_counterexample}

The first two figures concern the counterexample. The first compares computed
trajectories with the analytic regions, while the second compares the numerical
integrability residual with its closed form and three controls.

Figure~\ref{fig:annulus} compares six computed trajectories with the analytic
regions derived in Section~\ref{sec:counterexample}. The exact dynamics give an
inward spiral for $\|x\|<0.8$ and $\dot x=-(\lambda\alpha/2)Jx$ for
$0.8\le\|x\|\le r_\alpha$, where $r_\alpha=1.1/\sqrt{1+\alpha^2/4}$. The annular
period is $4\pi/\alpha=80\pi$. The trajectories are computed to time $520$, slightly
beyond two periods. On the three sampled annular trajectories the maximum radial
error is $4.3\times10^{-13}$ and the maximum phase error is $3.5\times10^{-13}$.

\begin{figure}[htbp]
\centering
\includegraphics[width=0.72\textwidth]{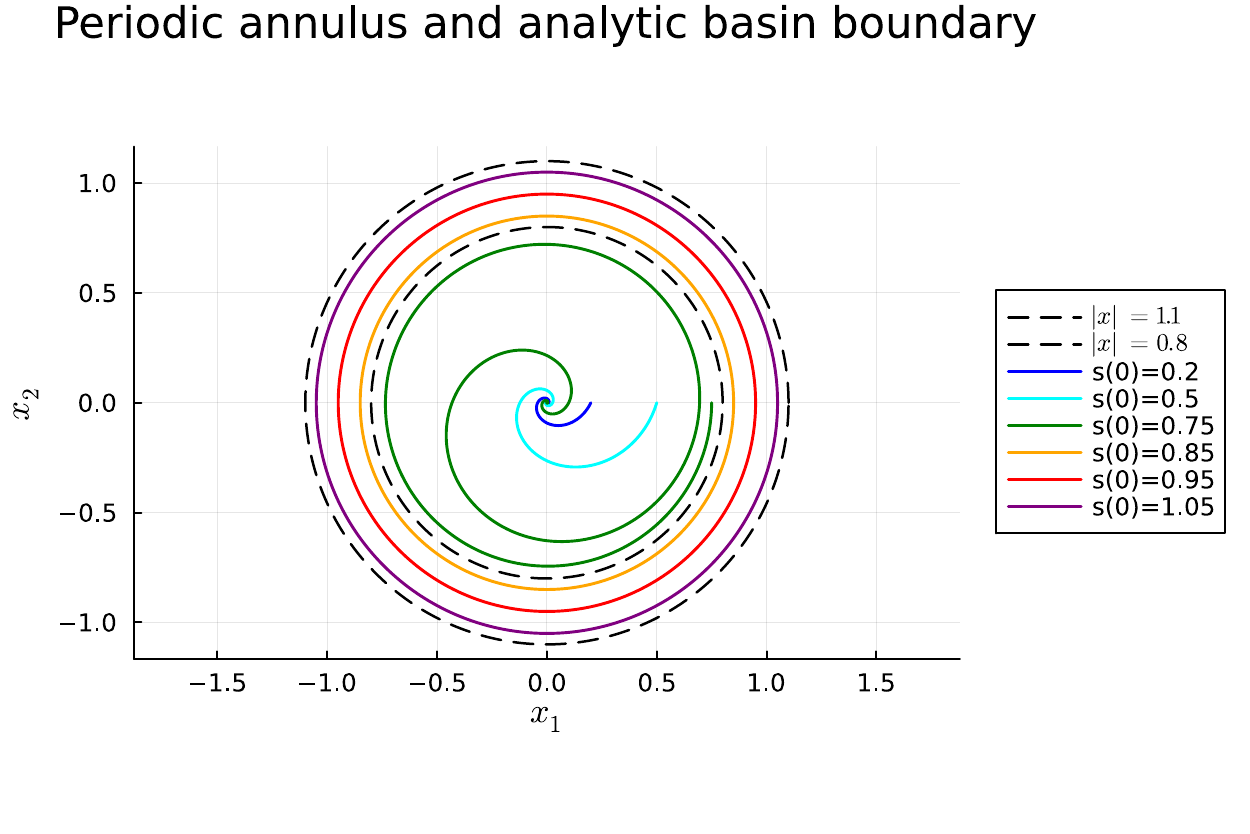}
\caption{Trajectories of \eqref{eq:SD_dynamics} for the metric of
Example~\ref{rem:periodic}, from initial radii $0.2$, $0.5$, $0.75$, $0.85$, $0.95$
and $1.05$ at angle zero. The dashed circles are $\|x\|=0.8$ and the boundary of
$\setS$. The periodic band and the basin boundary are analytic, and the trajectories
are numerically integrated.}
\label{fig:annulus}
\end{figure}

Figure~\ref{fig:residual} checks the numerical residual calculation against its
closed form and three controls.
Equation~\eqref{eq:counterexample_residual} gives the residual
$-(2\sqrt5/\|x\|)(\cos(\theta+\psi_0),\sin(\theta+\psi_0))$, of norm $2\sqrt5/\|x\|$,
and it is that closed form, not the grid, which shows the residual has no zeros.
Central differences with Richardson extrapolation confirm the sampled values. The raw
differences show the expected order close to two, with median observed order
$2.0007$, and the extrapolate is fourth order. On the sampled band
$0.85\le\|x\|\le1.05$ the residual norm runs from $4.2592$ to $5.2613$, against the
range $4.07$ to $5.59$ over the full annulus. The separable control returns zero,
while the bounded Hessian control has maximum residual norm $7.8\times10^{-12}$. For
the nonzero control the maximum vector error relative to $(0,x_2/2)$ is
$5.7\times10^{-12}$. This rules out an implementation that returns zero identically.

\begin{figure}[htbp]
\centering
\includegraphics[width=0.86\textwidth]{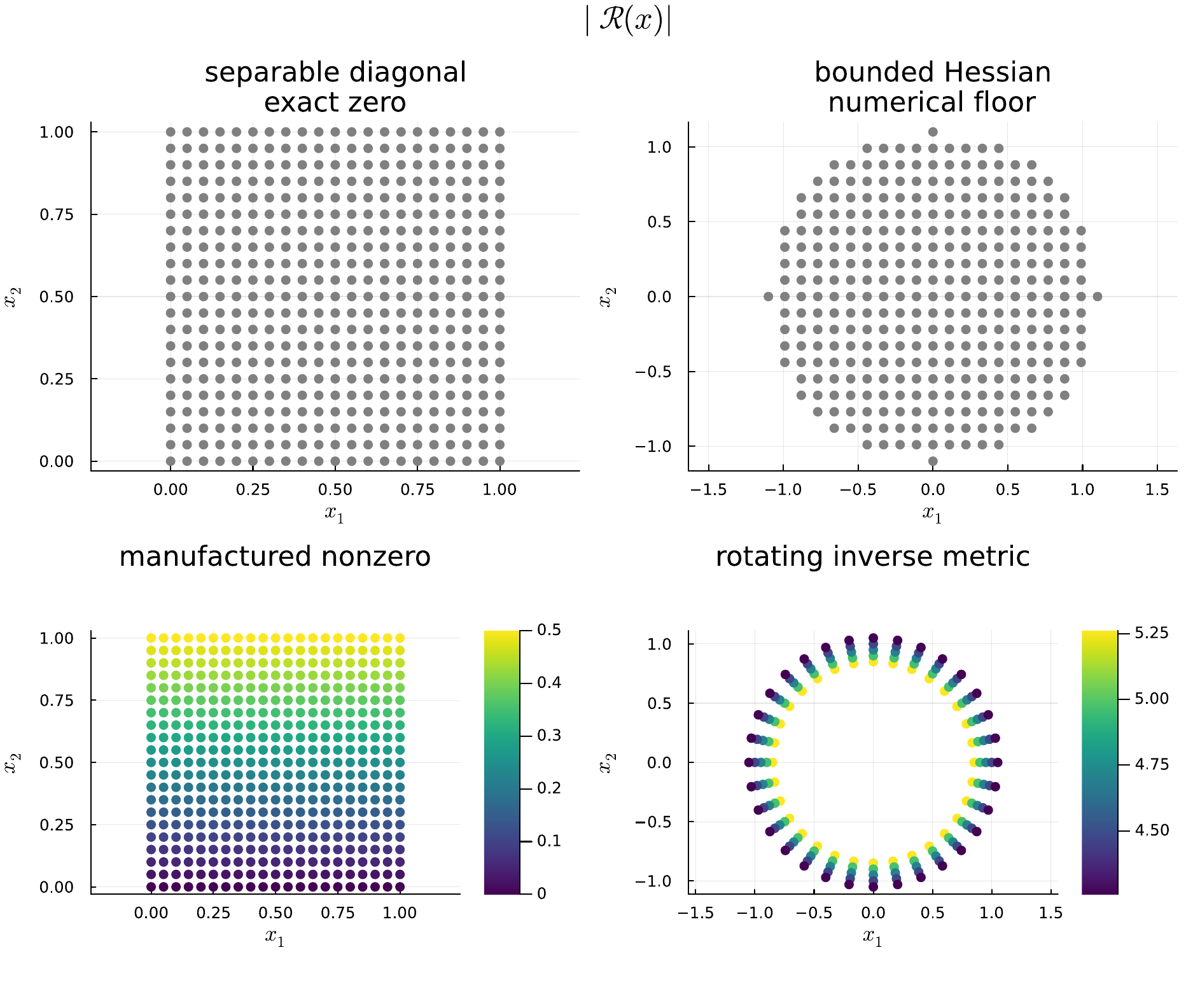}
\caption{Integrability residual $\partial_iH_{jk}-\partial_jH_{ik}$ evaluated by
Richardson-extrapolated central differences. Left to right, the separable diagonal
metric \eqref{eq:separable_metric}, the bounded Hessian metric, the metric
$H(x)=\diag(1+x_2^2/4,\,1)$ whose residual is $(0,x_2/2)$, and the rotating metric of
Example~\ref{rem:periodic} on $0.85\le\|x\|\le1.05$. The first two are integrable and
the last two are not.}
\label{fig:residual}
\end{figure}

\subsection{The two certificates}\label{sec:num_certificates}

The remaining two figures evaluate the certificates on the same family. The first
instantiates the local statement at $\tau=1$, the second traces the global one across
$\tau$.

Figure~\ref{fig:radius} illustrates the local half-rate certificate at $\tau=1$. Here
$C_L=42.714919$, $R_{1/2}=3.8682\times10^{-5}$ and
$\eta_{R_{1/2}}=1.6523\times10^{-3}$. A trajectory started at weighted radius
$R_{1/2}/2$ has no positive envelope excess at the $2001$ saved times on $[0,2000]$.
The maximum difference between the weighted-radius traces from the $10^{-10}$ and
$10^{-12}$ runs is $1.0\times10^{-17}$. A trajectory started at $\|x(0)\|=0.5$ lies outside this
certificate but inside the analytic basin. It illustrates certificate slack and does
not estimate the basin.

\begin{figure}[htbp]
\centering
\includegraphics[width=0.86\textwidth]{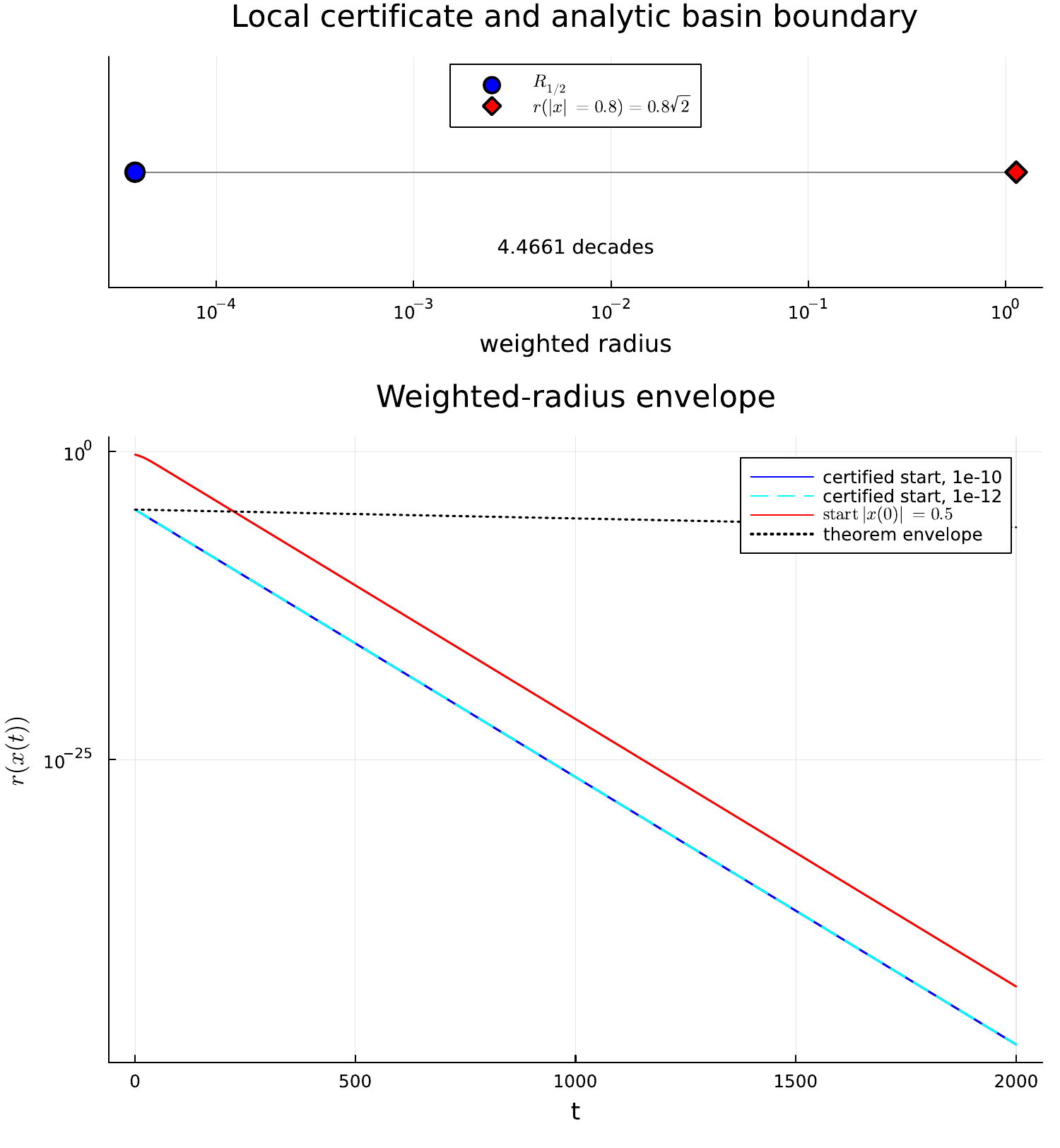}
\caption{The local certificate at $\tau=1$. Above, the half-rate threshold
$R_{1/2}=3.8682\times10^{-5}$ and the weighted radius $0.8\sqrt2$ of the analytic
basin boundary. Below, the weighted radius
$r(x(t))=\|x(t)-\bar x\|_{Q(x(t))}$ of the certified trajectory, its theorem
envelope, and a trajectory started at $\|x(0)\|=0.5$.}
\label{fig:radius}
\end{figure}

Figure~\ref{fig:band} evaluates the global certificate on \eqref{eq:interp_family}.
With $K_M\le3.571\tau$ the lower bound for $c_g$ is positive for $\tau<\tau_*$. At
$\tau=0.1$ it equals $0.014531$, while the residual at $\|x\|=0.95$ equals
$0.124209$. Thus the sufficient condition holds for a non-Hessian metric. Both panels
are analytic.

\begin{figure}[htbp]
\centering
\includegraphics[width=0.78\textwidth]{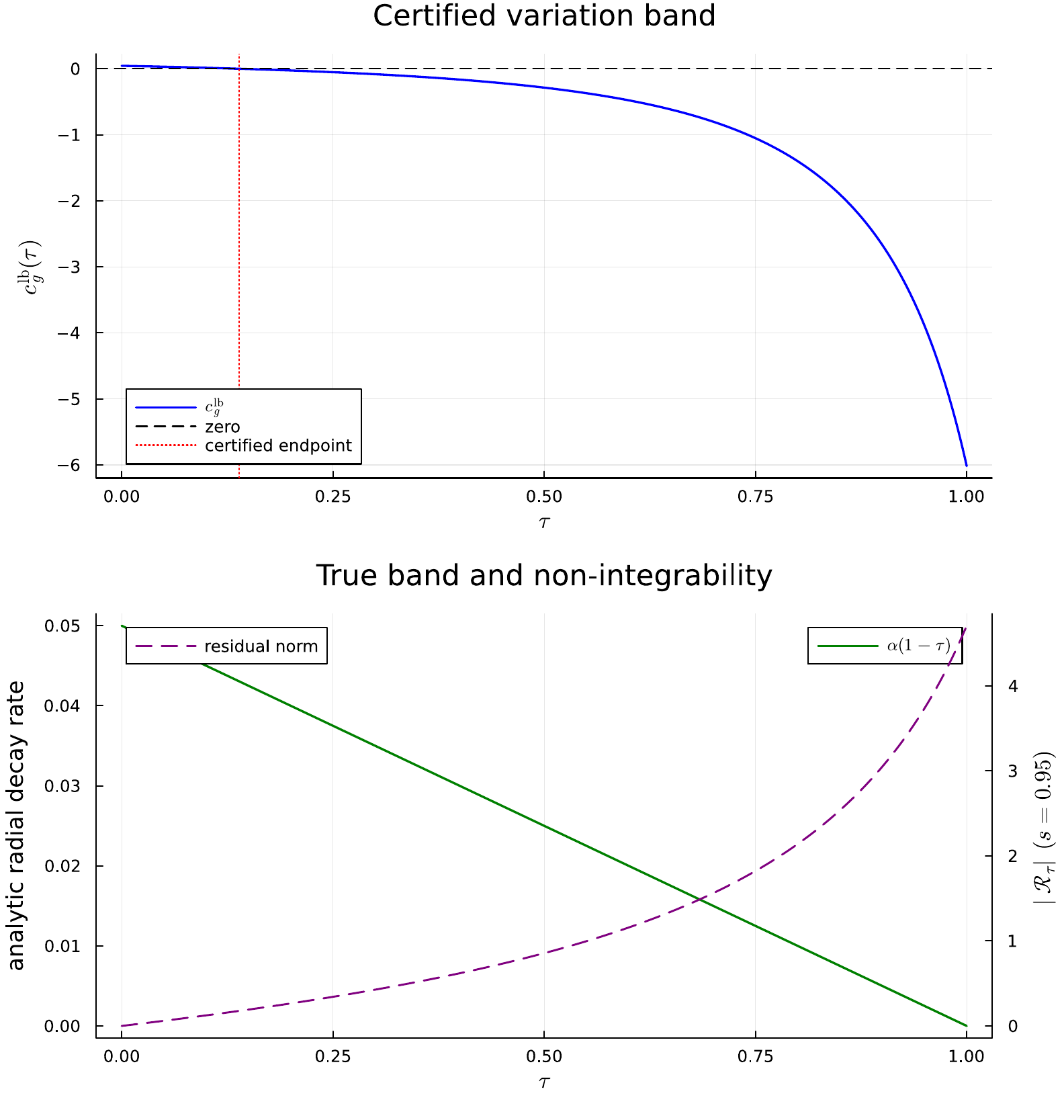}
\caption{The family \eqref{eq:interp_family}. Above, the certified lower bound for
$c_g$ and its root $\tau_*=0.1393$. Below, the exact radial decay rate
$\alpha(1-\tau)$ on $0.8\le\|x\|\le r_\alpha$ and the integrability residual at
$\|x\|=0.95$. Both panels are
analytic evaluations.}
\label{fig:band}
\end{figure}

The two displayed certificates have substantial slack on this family. At $\tau=1$ the
analytic basin boundary has weighted radius $0.8\sqrt2$, since $\chi=1$ there and
$(B^{-1})_{11}=2$. The half-rate threshold $R_{1/2}=3.8682\times10^{-5}$ is $4.4661$
decades below that boundary. The same constants permit every
$R<2R_{1/2}=7.7364\times10^{-5}$. As $R$ approaches this supremum $\eta_R$ approaches
zero, and the gap is $4.1651$ decades. For the global certificate the lower bound
obtained from $K_M\le3.571\tau$ is positive for $\tau<\tau_*$. For $F(x)=(I+2J)x$ the
family converges on the whole range $\tau<1$. Write $s=\|x\|$, $q=\tau\chi(s)$,
$\beta=1-q$ and $D_q=(1-q)I+qB$. In the frame of Example~\ref{rem:periodic} one has
$D_q(1,2)^\top=(\beta,(3\beta+1)/2)^\top$, so
$y=s[(1-\alpha\beta)e_r-\tfrac{\alpha}{2}(3\beta+1)e_\theta]$. Where the projection is
inactive this gives $\dot s=-\lambda\alpha\beta s$. Where it is active, the optimality
condition $Q(z-y)=-\nu z$ for $z=P_{\setS,Q}(y)$ becomes $(I+\nu D_q)z=y$, and
inverting the $2\times2$ matrix gives
\begin{equation}\label{eq:radial_active}
s(1-\alpha\beta)-\inner{z}{e_r}
=\frac{\nu s\bigl[(1-\alpha\beta)(1+\nu\det D_q)
+\tfrac{\alpha q}{4}(3\beta+1)\bigr]}{\det(I+\nu D_q)}\ \ge\ 0 .
\end{equation}
Both cases give $\dot s\le-\lambda\alpha(1-\tau)s$ on $\setS$, since $\chi\le1$. Every
trajectory converges exponentially for $\tau<1$, while at $\tau=1$
Section~\ref{sec:counterexample} supplies the periodic annulus. The endpoint is
exactly $1$ against $\tau_*$ for the certificate, a ratio of
$1/\tau_*=7.1763$. These values quantify this instance under the stated certified
bounds. They do not measure the sharpness of Theorem~\ref{thm:exponential} or of
condition~\eqref{eq:cond19}.

\subsection{The projection with an active constraint}\label{sec:num_active}

In every instance above the forward point stays inside $\setS$, so the projection
returns it unchanged and $\nu=0$. The estimates of
Section~\ref{sec:projreg} are then trivial, because the anchor
$\|P_{\setS,Q}(y)-y\|$ in \eqref{eq:proj_metric_var} vanishes. The following
instance moves the solution to the boundary.

Take $F(x)=x-c$ with $c=(2,0)$ outside $\setS$, keeping
$\setS$, $\lambda$ and $\alpha$ as before and using the metric
\eqref{eq:interp_family}. The variational inequality
$\inner{\bar x-c}{x-\bar x}\ge0$ is the characterisation of the Euclidean
projection of $c$, so $\bar x=1.1\,c/\|c\|=(1.1,0)$ lies on the boundary, for every
metric. Here $\mu=K=1$ and $F_{\max}=1.1+\|c\|=3.1$, so the constants of
Table~\ref{tab:constants} do not carry over.

At $\bar x$ the multiplier is available in closed form. The optimality condition
gives $Q(\bar x-y)=-\nu\bar x$, and $Q\matM=I$ turns this into
$\alpha F(\bar x)=-\nu\bar x$, so
\begin{equation}\label{eq:nu_exact}
\nu(\bar x)=\frac{\alpha\|F(\bar x)\|}{1.1}=\frac{9}{220}
=0.0409090909\ldots ,
\end{equation}
independently of $\tau$ and of the metric. This is
Lemma~\ref{lem:cancellation} in dynamical form, and it is the reference value
against which the solver is checked.

\begin{table}[htbp]
\centering
\caption{The projection at an active constraint, evaluated at $\bar x$ and at eight
further boundary points of the active arc for each $\tau$. Here $y$ is the forward
point, $z:=P_{\setS,Q}(y)$ its projection, $\nu$ the multiplier, and $\|z-y\|$ the
displacement anchoring \eqref{eq:proj_metric_var}. The two ratio columns give the
left side of \eqref{eq:proj_joint} divided by its right side, first with the
constant $L$ of \eqref{eq:proj_L} and then with $L^2$ in its place. A ratio below
one means the estimate holds. Each row summarizes nine direct pointwise projection
solves, with no differential equation integrated.}
\label{tab:activeproj}
\begin{tabular}{lccccc}
\toprule
$\tau$ & $\nu(\bar x)$ & $\|z-y\|$ at $\bar x$ & max ratio, $L$ & max ratio, $L^2$ & max KKT residual \\
\midrule
$1$   & $9/220$ & $0.050312$ & $0.159$ & $0.035$ & $8.3\cdot10^{-14}$ \\
$0.1$ & $9/220$ & $0.045056$ & $0.891$ & $0.819$ & $2.4\cdot10^{-14}$ \\
\bottomrule
\end{tabular}
\end{table}

The computed multiplier agrees with \eqref{eq:nu_exact} to twelve digits at both
values of $\tau$, and $\|T(\bar x)-\bar x\|\le2\times10^{-14}$. A second,
independently coded solve of the scalar equation differs by at most
$7\times10^{-14}$. The four KKT residuals are not by themselves
evidence, since with one smooth constraint three of them hold by construction, and
an implementation that exchanged $Q$ and $\matM$ would satisfy all four. That
exchange was run deliberately and is detected by \eqref{eq:nu_exact} and by the
fixed-point identity.

All sixteen sampled pairs satisfy \eqref{eq:proj_joint}. At $\tau=0.1$ the ratio
reaches $0.891$, so the estimate is close to attained there, and replacing $L$ by
$L^2$ lowers the ratio to $0.819$. At $\tau=1$ the corresponding figures are $0.159$
and $0.035$. This is the numerical counterpart of
Remark~\ref{rem:proj_constant_comparison}.

Together the computations reproduce the analytic trajectory and residual benchmarks,
instantiate the two sufficient conditions, and evaluate the projection estimates at
an active constraint. The computations do not establish the periodic annulus or
non-integrability, and they do not estimate a threshold.

\section{Limitations and Further Work} \label{sec:futurework}

For this finite-step flow, the closest general-matrix convergence analyses assume
that the inverse metric is the Hessian of a strongly convex function and use the
associated Bregman distance as a Lyapunov function
\cite{Amochkina1997,Mijajlovic2018}. Mijajlović and Jaćimović also obtain an
exponential estimate for a state-dependent scalar inverse metric
\cite{Mijajlovic2018}. Alvarez, Bolte and Brahic explain the role of the Hessian
hypothesis. It is exactly what makes the displacement fields gradients in the
metric \cite{Alvarez2004}. Here we study what stability can be proved for a
general matrix metric without that integrability structure.

The answer has three parts. A solution remains a fixed point of the projected map formed with the metric at every state, not only at the solution, which is what allows a Lyapunov argument in the current metric at all. Local exponential convergence survives, with an explicit rate and an explicit certified radius, and a global rate follows under a bound on the variation of the metric. Section~\ref{sec:counterexample} gives a compact feasible set, a strongly monotone operator, and a $C^2$ uniformly positive definite metric whose flow has an annulus of periodic orbits. Remark~\ref{rem:periodic_consistency} explains that this does not make condition~\eqref{eq:cond19} necessary. The construction shows only that the standing assumptions do not imply global convergence.

A natural discrete counterpart has the form
\[
x_{k+1} = P_{\setS,M_k^{-1}}\bigl(x_k - \alpha_k M_k F(x_k)\bigr), \qquad M_k \approx \matM(x_k).
\]
With a line search, this connects to the variable-metric gradient projection
process of \cite{Gawande1988}. The scaled gradient projection method computes the
same projected point and then applies
$x_{k+1}=x_k+\lambda_k(y_k-x_k)$ with a nonmonotone line search
\cite{Bonettini2009}; this update is an explicit Euler step for
\eqref{eq:SD_dynamics} with the metric frozen at $M_k$. At a constant metric, the
related relaxed forward--backward scheme is covered by
\cite{Bolte2003,Abbas2015}. An open question is whether state-dependent choices
$M_k\approx\matM(x_k)$ admit convergence guarantees analogous to the
continuous-time variation bounds proved here. The same question arises for
splitting and forward--backward schemes with evolving preconditioners.

Extending the local analysis to pseudomonotone or quasimonotone operators would
require replacing strict Lyapunov decay, for example by an invariance or
dissipativity argument that still controls metric variation.

Another direction concerns the choice of the state-dependent metric $\matM(\cdot)$. In the present analysis, $\matM(\cdot)$ is assumed to be prescribed and to satisfy uniform boundedness and Lipschitz regularity, but is otherwise arbitrary. This leaves considerable freedom in choosing metrics that reflect problem structure, local conditioning, or trajectory information. Future work may investigate adaptive or data-driven strategies for metric selection, for instance by coupling the SD-SPNN dynamics with estimators of curvature or sensitivity, or by restricting $\matM(x)$ to structured classes (diagonal, low-rank, block-diagonal). Whether such choices can be made while retaining the regularity the analysis requires is open.

Three limitations remain.

The Lipschitz regularity and uniform boundedness assumptions on the metric, while natural for Lyapunov-based analysis, may be restrictive in applications where nonsmooth or abruptly changing preconditioners are desirable.

The trajectory is required to start inside $\setS$. Lemma~\ref{lem:existence} gives viability from $x(0)\in\setS$ and says nothing otherwise. Both the Euclidean gradient-projection flow and the safe monotone flow are well defined from infeasible initial conditions. Antipin proves convergence from any initial point in $\R^n$, under a step size determined by the gradient Lipschitz constant \cite[Theorem~2]{Antipin1994}. Bolte obtains weak convergence for an infeasible start in Hilbert space using a different Lyapunov functional, the feasible descent estimate being unavailable outside the set \cite[Theorem~3.2]{Bolte2003}. For the safe monotone flow the restricted tangent set is nonempty on a neighbourhood of $\setS$ \cite{Allibhoy2025}. Extending the state-dependent flow in that direction is open.

Finally, the metric projection $P_{\setS,\matM(x)^{-1}}$ is an optimization subproblem at every point, and its cost depends on the structure of $\matM(x)$ and on the tractability of that subproblem. Competing non-Euclidean dynamics avoid this where the geometry is chosen for closed-form convenience, as with the entropy kernels that give the logit map on the simplex and the logistic map on the cube \cite{Mertikopoulos2018}. A free metric buys design freedom and pays for it here. Addressing this, for example through structure-exploiting metrics, approximate projections, or inexact dynamics, is an essential step toward large-scale applications.

\section*{Declarations}

\textbf{Conflict of interest:} The author declares that he has no conflict of interest.

\section*{Data and Code Availability}
No external data were used in this study. The code and numerical results are available from the corresponding author upon reasonable request.

\section*{Funding}
This research did not receive any specific grant from funding agencies in the public, commercial, or not-for-profit sectors.

\section*{AI Use Declaration}
During the preparation of this work the author used Claude (Anthropic) in order to assist with manuscript editing, including tightening prose, verifying \LaTeX{} formatting, and checking internal consistency of cross-references and notation. After using this tool, the author reviewed and edited the content as needed and takes full responsibility for the content of the publication.

\section*{Acknowledgement}

The author gratefully acknowledges the institutional support provided by King Fahd University of Petroleum \& Minerals (KFUPM), and the support of the Interdisciplinary Research Center for Smart Mobility and Logistics (IRC-SML) at KFUPM, which facilitated the research environment in which this work was conducted. The author also thanks Professor Qamrul Hasan Ansari for his valuable feedback and insightful input during the development of this work.

\printbibliography

\end{document}